\documentclass[12pt]{article}
\usepackage{amsmath}
\usepackage{amssymb}
\usepackage{amsthm}
\usepackage{graphicx}
\usepackage{verbatim}
\usepackage{hyperref}
\numberwithin{equation}{section}

\newtheorem{thm}{Theorem}[section]
\newtheorem{cor}[thm]{Corollary}
\newtheorem{lem}[thm]{Lemma}

\begin{document}

\title{Volume growth and stochastic completeness of graphs}
\author{Matthew Folz\thanks{Department of Mathematics, The University of British Columbia, 1984 Mathematics Road, Vancouver, B.C., Canada, V6T 1Z2.  {\tt mfolz@math.ubc.ca}.  Research supported by a NSERC Alexander Graham Bell Canada Graduate Scholarship.}}
\maketitle

\begin{abstract}
\noindent  Given the variable-speed random walk on a weighted graph and a metric adapted to the structure of the random walk, we construct a Brownian motion on a closely related metric graph which behaves similarly to the VSRW and for which the associated intrinsic metric has certain desirable properties.  Jump probabilities and moments of jump times for Brownian motion on metric graphs with varying edge lengths, jump conductances, and edge densities are computed.  We use these results together with a theorem of Sturm for stochastic completeness, or non-explosiveness, on local Dirichlet spaces to prove sharp volume growth criteria in adapted metrics for stochastic completeness of graphs. \\

\end{abstract}

%\begin{small}
%{\bf Key words:} random walk, stochastic completeness, volume growth, adapted metric, Brownian motion on metric graphs. \\
%
%{\bf AMS 2010 Subject Classification:} Primary: 60G50; Secondary: 60J65, 31C25. \\
%\end{small}

\section{Introduction} \label{S1}

Let $\Gamma=(G,E)$ be an unoriented, connected, finite or countably infinite, locally finite graph.  We assume that $\Gamma$ has neither loops nor multiple edges.  We use $d$ to denote the graph metric on $\Gamma$; given $x,y\in G$, $d(x,y)$ is equal to the number of edges in a shortest (geodesic) path between $x$ and $y$. Given $x,y\in G$, we write $x\sim y$ if $\{x,y\}\in E$.  We assume that $\Gamma$ is a weighted graph, so that associated with each $(x,y)\in G\times G$ is a nonnegative edge weight $\pi_{xy}$ which is symmetric ($\pi_{xy}=\pi_{yx}$ for $x,y\in G$) and satisfies $\pi_{xy}>0$ if and only if $\{x,y\}\in E$.  The edge weights define a measure on $G$ by setting $\pi_x := \pi(\{x\}) := \sum_{y\in G} \pi_{xy}$ for $x\in G$, and extending to all subsets of $G$ by countable additivity.  If $\pi(e)=1$ for all $e\in E$, we say that $\Gamma$ has the standard weights.  We denote weighted graphs by the pairing $(\Gamma, (\pi(e))_{e\in E})$; for brevity we write this as $(\Gamma,\pi)$. \\

On $(\Gamma,\pi)$, we consider the variable-speed continuous time simple random walk (VSRW), which we denote by $(X_t)_{t\geq 0}$.  This process has generator $\mathcal{L}_X$ given by

\begin{align*}
(\mathcal{L}_X f)(x) := \sum_{y\sim x} \pi_{xy}(f(y)-f(x)).
\end{align*}
\\
The jump probabilities for $X$ are $P(x,y) = \pi_{xy}/\pi_x$; the jump times, started at a vertex $x$, are exponentially distributed random variables with parameter $\pi_x$, so that the mean jump time is $1/\pi_x$.  This process is strong Markov.  The heat kernel of $X$ is the function $p_t(x,y) := \mathbb{P}^x(X_t=y)$, defined for $t\geq 0$ and $x,y\in G$.  A good general reference for the VSRW is the upcoming book by Barlow \cite{B}.  For analytic properties of the operator $\mathcal{L}_X$, which is sometimes referred to as the physical Laplacian, see \cite{We}. \\

A weighted graph $(\Gamma,\pi)$ is stochastically complete if the heat kernel of the VSRW on $(\Gamma,\pi)$ satisfies

\begin{equation} \label{sccrit}
\sum_{y\in G} p_t(x,y) = 1
\end{equation}
\\
for all $t\geq 0$ and $x\in G$.  If $(\pi(e))_{e\in E}$ are the standard weights, and $(\Gamma,\pi)$ is stochastically complete, we simply say that $\Gamma$ is stochastically complete.  A weighted graph which is not stochastically complete is said to be stochastically incomplete.  The heat kernel condition \eqref{sccrit} is equivalent to $X$ having infinite lifetime (or $X$ being non-explosive).  If $(J^X_n)_{n\in\mathbb{Z}_+}$ are the times between jumps of $X$, then non-explosiveness is equivalent to having, $\mathbb{P}-$a.s.,

\begin{equation*}
\sum_{n\in\mathbb{Z}_+} J^X_n = +\infty.
\end{equation*}
\\
Stochastic completeness is also equivalent to various uniqueness criteria for the heat equation on $(\Gamma,\pi)$; see \cite{Wo}, and also \cite{G2} for stochastic completeness in the related setting of Riemannian manifolds. \\

This paper is primarily concerned with the relationship between volume growth and stochastic completeness for graphs.  Generically speaking, small volume growth ensures that the VSRW does not escape from its starting point too quickly, as it implies that the process, on average, is not too much more likely to move away from its starting point than towards it.  For diffusions on a Riemannian manifold, the best possible criterion relating volume growth in the Riemannian volume and stochastic completeness of the manifold was proved by Grigor'yan in \cite{G1} (see also \cite{G2} for a more detailed exposition of results on stochastic completeness for manifolds), and is as follows:

\begin{thm}
(Grigor'yan, \cite{G1}) Let $M$ be a geodesically complete Riemannian manifold with Riemannian metric $\rho_M$ and Riemannian volume $V$.  If there exists $r_0\geq 0$ and $x_0\in M$ such that

\begin{equation} \label{VGSCM}
\int^\infty_{r_0} \frac{r}{\log V(B_{\rho_M}(x_0,r))}dr = +\infty,
\end{equation}
\\
then $M$ is stochastically complete.
\end{thm}

This result was generalized to the setting of strongly local Dirichlet spaces by Sturm in \cite{St}; there, the Riemannian metric is replaced by the intrinsic metric associated with the Dirichlet form.  This result is not applicable to the setting of graphs, as the Dirichlet form associated with the VSRW is non-local. \\

Stochastic completeness of graphs has been a topic of substantial interest in recent years.  Attempting to modify the techniques used in the manifold setting to the setting of graphs quickly leads to difficulties, due in large part to the fact that the heat kernel of the VSRW has very different long range behavior than the heat kernel of a Riemannian manifold (see \cite{F1} and \cite{G1}).
The situation is further complicated by the fact that the Riemannian metric plays a different role for diffusions on manifolds than the graph metric does for the variable-speed random walk.  In order to obtain a criterion analogous to Grigor'yan's manifold result, it is necessary to consider metrics which are in some sense `adapted' to the structure of the VSRW; we will discuss the notion of adaptedness subsequently. \\

Previous results for stochastic completeness on graphs have focused primarily on relatively simple graphs possessing a high degree of symmetry.  Numerous results for spherically symmetric graphs and trees are contained in \cite{Wo}.  Necessary and sufficient criteria for stochastic completeness of weakly symmetric graphs are established in \cite{KLW}.  Stochastic completeness of subgraphs is discussed in \cite{KL}, and \cite{Hu1} establishes analytic criteria which are equivalent to stochastic completeness of graphs and which have analogues in the manifold setting.  Criteria for stochastic completeness in the general setting of symmetric jump processes are proved in \cite{MU}.  A proof of stochastic completeness in a particular adapted metric assuming at most exponential volume growth is given in \cite{B+2}. \\

The recent paper of Grigor'yan, Huang, and Masamune (\cite{GHM}), which also deals with symmetric jump processes, establishes a criterion for stochastic completeness on graphs if one considers volume growth with respect to the graph metric, as well as a volume growth criterion for stochastic completeness of graphs using adapted metrics.  This criterion is also obtained via different techniques in \cite{Hu2}.  However, this result is not analogous to the manifold criterion \eqref{VGSCM}, and does not produce optimal results when compared with the exact criteria in \cite{Wo} for spherically symmetric graphs or trees. \\

Our main result is as follows:

\begin{thm} \label{mainresult}
Let $(\Gamma,\pi)$ be a weighted graph, and let $\rho$ be a metric on $\Gamma$ satisfying the following two conditions:

\begin{itemize}
\item There exists $c_\rho>0$ such that $\rho(x,y) \leq c_\rho$ whenever $x\sim y$.
\item There exists $C_\rho>0$ such that 
\begin{equation*}
\sum_{y\sim x} \pi_{xy}\rho^2(x,y) \leq C_\rho,
\end{equation*}
\\
for each $x\in G$.  
\end{itemize}Denote by $B_\rho(x,r)$ the closed ball of radius $r$ centered at $x$ in the metric $\rho$, and let $|\cdot|$ denote counting measure. \\

If there exists $x_0\in G$ and $r_0\geq 0$ such that

\begin{equation*}
\int^\infty_{r_0} \frac{r}{\log|B_\rho(x_0,r)|}dr = +\infty,
\end{equation*}
\\
then $(\Gamma,\pi)$ is stochastically complete.
\end{thm}

An immediate consequence of this criterion is that a weighted graph is stochastically complete if there exists $x_0\in G$ and $r_0\geq 0$ such that that $|B_\rho(x_0,r)|\leq Ce^{cr^2\log r}$ for $r\geq r_0$.  This is a substantial improvement over the best previous result relating volume growth and stochastic completeness of graphs (found in \cite{GHM} and \cite{Hu2}), which proved stochastic completeness under the assumption that there exists $x_0\in G$, $r_0\geq 0$, and $c\in (0,2)$ such that $|B_\rho(x_0,r)|\leq Ce^{cr\log r}$ for $r\geq r_0$. \\

The criterion in Theorem~\ref{mainresult} will be seen to yield sharp results for stochastic completeness of graphs in certain applications where necessary and sufficient conditions for stochastic completeness are known, such as birth-death chains.  As Theorem~\ref{mainresult} is also analogous to the result of Grigor'yan relating volume growth with stochastic completeness on manifolds, one expects that Theorem~\ref{mainresult} is essentially the best possible criterion relating volume growth and stochastic completeness of graphs. \\

The condition on the metric in the hypotheses of Theorem~\ref{mainresult} appeared previously in work of the author in \cite{F1}, and has also been used in papers of Frank, Lenz, and Wingert (\cite{FLW}) and Grigor'yan, Huang, and Masamune (\cite{GHM}).  A closely related metric is the intrinsic metric arising from the Dirichlet space construction of the VSRW, which was studied first by Davies in \cite{D}.  In light of the aforementioned result of Sturm for local Dirichlet spaces, one might expect that an analogue of Theorem~\ref{mainresult} holds using the intrinsic metric for the VSRW; we will show by a counterexample that this is not the case. \\

While this result deepens our understanding of the relationship between continuous time simple random walks on graphs and diffusion processes on manifolds, the overall picture is still far from clear.  This paper shows that the sharp volume growth criteria for stochastic completeness on graphs and on manifolds are the same.  The author has also shown in \cite{F2} that the optimal volume growth criteria for the bottom of the spectrum of the Laplacian on graphs and manifolds are also very similar.  However, the recent paper of Huang \cite{Hu2} shows that a certain uniqueness class for the heat equation (with close connections to stochastic completeness) is very different for graphs and manifolds.  As well, it has been known for some time that heat kernel behavior is very different for graphs and manifolds; while the heat kernel of a diffusion on a manifold admits a Gaussian upper bound (see \cite{G3}), the heat kernel of a continuous time simple random walk only satisfies a Gaussian upper bound in a certain space-time region (see \cite{F1}), and satisfies a different, `Poisson-type' heat kernel estimate elsewhere (see \cite{D} and \cite{F1}).  It would be useful to understand the connections between these disparate results in some sort of unified framework. \\ 

The structure of the paper is as follows.  In Section~\ref{S2}, we define Brownian motion on metric graphs, and compute jump probabilities and moments of hitting times for Brownian motions on weighted metric graphs.  The Dirichlet form associated with these processes is strongly local, and consequently the result of Sturm relating volume growth in the intrinsic metric to stochastic completeness is applicable in this setting.  In Section~\ref{S3}, we discuss various metrics for the VSRW and for Brownian motion on metric graphs, with an emphasis on intrinsic metrics and adapted metrics.  These metrics are compared under varying hypotheses, and an example is given showing that on graphs with unbounded vertex degrees, the intrinsic metric for the VSRW may have certain undesirable properties from the perspective of stochastic completeness.  In Section~\ref{S4}, given the VSRW on a weighted graph and an adapted metric, we use the adapted metric to define a weighted metric graph which has the same vertex set as the weighted graph.  We show that Brownian motion on this metric graph behaves similarly to the VSRW with respect to jump probabilities and expected jump times, and that the intrinsic metric on the metric graph is larger than the adapted metric on the graph. Finally, in Section~\ref{S5}, we prove Theorem~\ref{mainresult}; under the hypotheses of this theorem, our comparison results for metrics, combined with Sturm's result for stochastic completeness on local Dirichlet spaces, imply that the Brownian motion on the associated metric graph defined in Section~\ref{S4} is non-explosive.  By using the estimates on various moments of hitting times for Brownian motion on metric graphs from Section~\ref{S2}, we show that this implies that the VSRW is also non-explosive.  We then present examples demonstrating that Theorem~\ref{mainresult} yields sharp results when applied to graphs for which necessary and sufficient conditions for stochastic completeness are known. \\

\section{Brownian motion on metric graphs} \label{S2}

This section contains a construction of Brownian motion on metric graphs, and computations of hitting probabilities and moments of hitting times for this process.  These results will be used in Section~\ref{S4} to construct a Brownian motion on a metric graph which behaves similarly to a given VSRW; this will be useful in our subsequent study of stochastic completeness of graphs.

\subsection{Construction} \label{S2.1}

We begin with a graph $\Gamma=(G,E)$.  For each $e\in E$, we assign a positive, finite edge length $\ell(e)$, and give this graph a metric structure by identifying the edge $e$ with a copy of the closed interval $[0,\ell(e)]$; the endpoints of $e$ are identified with the points $0$ and $\ell(e)$.  When we wish to emphasize the metric structure of the edge $e$ we denote it by $I(e)$.  The resulting object, which we denote by $\mathcal{X}(\Gamma,(\ell(e))_{e\in E})$, is a metric graph.  We define a geodesic metric on $\mathcal{X}(\Gamma,(\ell(e))_{e\in E})$: given $x,y\in \mathcal{X}(\Gamma,(\ell(e))_{e\in E})$, $\widetilde{d}(x,y)$ is the length of a shortest path (using the Euclidean metric along each edge) joining $x$ and $y$.  For more details, including defining a topological structure on $\mathcal{X}(\Gamma,(\ell(e))_{e\in E})$, see \cite{Ha}. \\

In the metric graph setting, we will want to consider graphs which have loops.  Given a graph $\Gamma=(G,E)$ with no loops, we denote by $\Gamma_{\text{loop}} := (G,E_{\text{loop}})$ the augmented graph with the same vertex set and and an edge set which consists of all the edges in $E$, plus loop edges added at every vertex of $G$.  We denote the loop incident to a vertex $x$ by $x_{\text{loop}}$.  Given a vertex $x\in G$, we write $E(x)$ to denote all non-loop edges incident to $x$ (regardless of whether we are working on $\Gamma$ or $\Gamma_{\text{loop}}$) and $E_{\text{loop}}(x)$ to denote all edges (loop or non-loop) incident to $x$. \\

We equip $\mathcal{X}(\Gamma,(\ell(e))_{e\in E})$ with two sets of edge weights: positive edge densities $(\omega(e))_{e\in E}$ and positive jump conductances $(p(e))_{e\in E}$.  We call $\mathcal{X}(\Gamma,(\ell(e))_{e\in E},(\omega(e))_{e\in E}, (p(e))_{e\in E})$ a weighted metric graph. \\

For brevity, we henceforth write

\begin{align*}
\mathcal{X}(\Gamma,\ell) &:= \mathcal{X}(\Gamma,(\ell(e))_{e\in E}), \\
\mathcal{X}_{\text{loop}}(\Gamma,\ell) &:= \mathcal{X}(\Gamma_{\text{loop}},(\ell(e))_{e\in E_{\text{loop}}}), \\
\mathcal{X}(\Gamma,\ell,p,\omega) &:= \mathcal{X}(\Gamma,(\ell(e))_{e\in E},(\omega(e))_{e\in E}, (p(e))_{e\in E}), \\
\mathcal{X}_{\text{loop}}(\Gamma,\ell,p,\omega) &:= \mathcal{X}(\Gamma_{\text{loop}},(\ell(e))_{e\in E_{\text{loop}}},(\omega(e))_{e\in E_{\text{loop}}}, (p(e))_{e\in E_{\text{loop}}}).
\end{align*}
\\
We will construct Brownian motion on weighted metric graphs using Dirichlet space theory.  For simplicity, we work for now on $\mathcal{X}(\Gamma,\ell,p,\omega)$, although the case of loops is handled by simply changing all sums to be over $E_{\text{loop}}$ instead of $E$.  We consider the Hilbert space $L^2(\mathcal{X}(\Gamma,\ell,p,\omega),\mu)$, where

\begin{equation*}
\mu(dx) := \sum_{e\in E} {\bf 1}_{I(e)}(x)p(e)\omega(e)m(dx),
\end{equation*}
\\
and $m$ is Lebesgue measure on the edge $I(e)$. \\

Let $C(\mathcal{X}(\Gamma,\ell))$ denote the set of continuous functions on our weighted metric graph; continuity clearly does not depend on the choice of the edge weights $(\omega(e))_{e\in E}$ or $(p(e))_{e\in E}$.  Similarly, let $C_c(\mathcal{X}(\Gamma,\ell))$ denote the set of continuous functions with compact support, and let $C_0(\mathcal{X}(\Gamma,\ell))$ denote the closure of $C_c(\mathcal{X}(\Gamma,\ell))$ with respect to the $\|\cdot\|_\infty$ norm. \\

For $k\in\mathbb{Z}_+$ and $p\geq 1$, set

\begin{align*}
\mathcal{S}^{k,p}(\mathcal{X}(\Gamma,\ell,p,\omega),\mu) &:= \left\{u\in C(\mathcal{X}(\Gamma,\ell)):\text{for all $e\in E$}, u\big|_{I(e)}\in W^{k,p}(I(e),\mu\big|_{I(e)})\right\}, \\
W^{k,p}(\mathcal{X}(\Gamma,\ell,p,\omega),\mu) &:= \left\{u\in \mathcal{S}^{k,p}(\mathcal{X}(\Gamma,\ell,p,\omega),\mu):\sum_{e\in E} \|u\|^p_{W^{k,p}(I(e),\mu\big|_{I(e)})} < \infty\right\}, \\
W^{k,p}_0(\mathcal{X}(\Gamma,\ell,p,\omega),\mu) &:= C_0(\mathcal{X}(\Gamma,\ell)) \cap W^{k,p}(\mathcal{X}(\Gamma,\ell,p,\omega),\mu).
\end{align*}

We consider the Dirichlet form on $L^2(\mathcal{X}(\Gamma,\ell,p,\omega),\mu)$ with domain $\mathcal{D}(\mathcal{E}) = W^{1,2}_0(\mathcal{X}(\Gamma,\ell,p,\omega),\mu)$, which is defined on $\mathcal{D}(\mathcal{E})\times\mathcal{D}(\mathcal{E})$ by

\begin{equation*}
\mathcal{E}(f,g) := \sum_{e\in E}\int_{I(e)} f'(x)g'(x)p(e)m(dx).
\end{equation*}
\\
The closedness of the form $\mathcal{E}$ is implied by the completeness of the Sobolev spaces $W^{k,p}(I)$, where $I$ is a finite interval. \\

The process associated with the Dirichlet form on this space is Brownian motion on $\mathcal{X}(\Gamma,\ell,p,\omega)$, which we denote by $(Y_t)_{t\geq 0}$.  From this construction, $(Y_t)_{t\geq 0}$ is a Hunt process, and consequently  $(Y_t)_{t\geq 0}$ is strong Markov. \\

By the Gauss-Green formula, one may recover conditions for the generator of $(Y_t)_{t\geq 0}$, $\mathcal{L}_Y$; the domain of $\mathcal{L}_Y$ is the set of functions $u$ such that $u\in W^{2,2}_0(\mathcal{X}(\Gamma,\ell,p,\omega),\mu)$ and, for each $x\in G$,

\begin{equation} \label{compcond}
\sum_{e\in E(x)} p(e)u'_e(x) = 0,
\end{equation}
\\
where $u'_e(x)$ denotes the directional derivative of $u$ at $x$ in the direction of the edge $e$.  On each edge $I(e)$, the generator $\mathcal{L}_Y$ is given by 

\begin{equation*}
\mathcal{L}_Y\Big|_{I(e)} u\Big|_{I(e)} = \frac{1}{2\omega(e)}u\Big|_{I(e)}''.
\end{equation*}
\\
One may view the Brownian motion on $\mathcal{X}(\Gamma,\ell,p,\omega)$ as being obtained from a single standard Brownian motion on $\mathbb{R}$, $(W_t)_{t\geq 0}$, in the following way.  We start the Brownian motion at some vertex $x\in G$.  By a standard result for Brownian motion on $\mathbb{R}$, the excursion set $\{W_t\not=0\}$ consists of countably many intervals.  For each excursion, we pick an adjacent edge $f$ with probability $p(f)(\sum_{e\in E(x)} p(e))^{-1}$, and move like $|W_{\omega(f)^{-1}t}|$ along this edge.  This is done until we reach a new vertex, at which point we use the strong Markov property and continue. \\

We discuss briefly on the changes necessary to consider Brownian motion on $\mathcal{X}_{\text{loop}}(\Gamma,\ell,p,\omega)$ instead of $\mathcal{X}(\Gamma,\ell,p,\omega)$.  The definition via Dirichlet spaces goes through unchanged if one replaces $E$ with $E_{\text{loop}}$.  For the generator, the compatiblity condition \eqref{compcond} becomes

\begin{equation} \label{compcondloop}
\sum_{e\in E(x)} p(e)u'_e(x) + p(x_{\text{loop}})u'_1(x_\text{loop})+p(x_{\text{loop}})u'_2(x_\text{loop})= 0.
\end{equation}
\\
where $u'_1(x_\text{loop}), u'_2(x_\text{loop})$ are the directional derivatives of $u$ at $x$ in the direction of either of the two possible orientations of the loop edge $x_{\text{loop}}$.  Similarly, if one views the Brownian motion on $\mathcal{X}_{\text{loop}}(\Gamma,\ell,p,\omega)$ as being obtained from a single standard Brownian motion $(W_t)_{t\geq 0}$, then for each interval, we pick a non-loop edge $f$ with probability $p(f)(\sum_{e\in E(x)} p(e)+2p(x_{\text{loop}}))^{-1}$, and each of the two possible orientations of the loop edge is chosen with probability $p(x_{\text{loop}})(\sum_{e\in E(x)} p(e)+2p(x_{\text{loop}}))^{-1}$. \\

The study of Brownian motion on metric graphs began with the paper of Walsh \cite{Wa}, and deals with the particular case where $x$ is a vertex connected to some (possibly infinite) number of finite or infinite rays emanating from $x$.  Various constructions of this process have been given; by Rogers \cite{R} (using resolvents), by Baxter and Chacon \cite{BC} (using the generator), by Varopoulos \cite{V} (using Dirichlet form theory), and by Salisbury \cite{Sa} (using excursion theory).  The paper of Varopoulos is of particular interest, as it discusses the present setting of Brownian motion on metric graphs (with finite edge lengths), and relates this to the study of continuous-time Markov chains.  Recently, there have been several papers by Kostrykin, Potthoff, and Schrader  (see in particular \cite{K+1}, \cite{K+2}) discussing Brownian motions on metric graphs; they consider a more general boundary condition at vertices than the one imposed here, and they allow edges of infinite length. \\

Brownian motion on metric graphs was initially studied as an object in the theory of diffusions, and since has been examined as an analytic object in its own right.  It has been used on previous occasions as a tool in the study of random walks; Varopoulos' paper \cite{V} establishes estimates for discrete time simple random walks using Brownian motion on metric graphs, and the paper of Barlow and Bass (\cite{B+1}) uses Brownian motion on metric graphs to study a continuous-time simple random walk.  From an analytic perspective, Brownian motion on metric graphs has many desirable properties, due to the panoply of existing results on one-dimensional Brownian motion, as well as the applicability of general results on diffusions or local Dirichlet spaces. \\

\subsection{Hitting probabilities and hitting times} \label{S2.2}

This section collects several results on the hitting probabilities and hitting times of Brownian motion on metric graphs.  We will work exclusively with the augmented weighted metric graph $\mathcal{X}_{\text{loop}}(\Gamma,\ell,p,\omega)$.  We set

\begin{equation*}
q(e) := \begin{cases} p(e) &\text{if $e$ is not a loop edge}, \\ 2p(e) &\text{if $e$ is a loop edge}. \end{cases}
\end{equation*}
\\
Using this definition, \eqref{compcondloop} becomes

\begin{equation} \label{compcondloop2}
\sum_{e\in E(x)} q(e)u'_e(x) + \frac{1}{2}q(x_{\text{loop}})u'_1(x_\text{loop})+\frac{1}{2}q(x_{\text{loop}})u'_2(x_\text{loop})= 0.
\end{equation}
\\
Our results for hitting probabilities and hitting times will be stated strictly in terms of $q$ because Theorem~\ref{BMjump1} and Theorem~\ref{BMjump2} are simpler to express in terms of the $q$ conductances.  \\

Denote Brownian motion on $\mathcal{X}_{\text{loop}}(\Gamma,\ell,p,\omega)$ by $(Y_t)_{t\geq 0}$.  As we will not be using the VSRW in this section, there is no ambiguity in writing $\mathcal{L} := \mathcal{L}_Y$. \\

Let the vertex $x$ be adjacent to the vertices $x_1,\ldots,x_k$ via the edges $e_1,\ldots e_k$, and to itself through the loop $x_{\text{loop}}$.  We use $B$ to denote the part of the graph containing $x,x_1\ldots,x_k$, as well as the metric edges $I(e_1),\ldots,I(e_k),I(x_{\text{loop}})$.  We also introduce coordinates on $B$, rooted at $x$; given $1\leq j\leq k$ and $0\leq z\leq \ell(e_j)$, we write $(j,z)$ for the unique point on $I(e_j)$ which is at a distance of $z$ from $x$.  On the loop, we choose an arbitrary orientation and write $(\text{loop},z)$ for the unique point on $I(x_{\text{loop}})$ which is at a positive distance of $z$ from $x$ (according to the orientation of the loop). \\

We write $T=T_Y:=\inf\{t\geq 0:Y_t\not\in B\}$ for the first time that $(Y_t)_{t\geq 0}$, started at $y$, hits any of $x_1,\ldots,x_k$.  Given $y\in B$, we write $\mathbb{P}^{y}$ for the probability law of $(Y_t)_{t\geq 0}$ started at $y$, and $\mathbb{E}^y$ for expectations with respect to this probability law. \\

We remark also that all results in this section are valid for loopless metric graphs; one needs only to replace all sums over $E_{\text{loop}}(x)$ with sums over $E(x)$.

\begin{thm} \label{BMprob}
For $1\leq j\leq k$,

\begin{equation*}
\mathbb{P}^{x}(Y_T = x_j) = \frac{\frac{q(e_j)}{\ell(e_j)}}{\sum_{e\in E(x)} \frac{q(e)}{\ell(e)}}.
\end{equation*}
\end{thm}

\begin{proof}
Let $u(y) := \mathbb{P}^y(Y_T=x_j)$.  By general Markov process theory, $u$ satisfies

\begin{equation*}
\begin{cases} \mathcal{L}u = 0 &\text{in }B, \\ u = 1 &\text{at } x_j, \\ u=0 &\text{on }\partial B\setminus\{x_j\}.\end{cases}
\end{equation*}
\\
By solving the differential equation along each edge, we obtain constants $(a_i)_{1\leq i\leq k}$ and $(b_i)_{1\leq i\leq k}$ such that $u((i,y)) = a_iy+b_i$.  By letting $y\downarrow 0$, we get that $b_i=b$ for $1\leq i\leq k$.  On the loop, $u$ is linear and satisfies $u((\text{loop},0))=u((\text{loop},\ell(x_{\text{loop}})))=b$, so $u$ is identically equal to $b$ along the loop. \\

Plugging in the boundary condition shows that, for $1\leq i\leq k$,

\begin{equation} \label{eq1}
\begin{cases} a_i\ell(e_i)+b = 1 &\text{if }i=j, \\ a_i\ell(e_i)+b = 0&\text{if }i\not=j.\end{cases}
\end{equation}
\\
By \eqref{compcondloop2}, since the terms corresponding to the loop vanish, we have

\begin{equation*}
\sum^k_{i=1} q(e_i)a_i =0.
\end{equation*}
\\
By multiplying the $i-$th equation in $\eqref{eq1}$ by $\frac{q(e_i)}{\ell(e_i)}$ and summing over $1\leq i\leq k$, we obtain that

\begin{equation*}
\sum^k_{i=1} q(e_i)a_i+b\sum^k_{i=1} \frac{q(e_i)}{\ell(e_i)} = \frac{q(e_j)}{\ell(e_j)},
\end{equation*}
\\
and hence

\begin{equation*}
\mathbb{P}^{x}(Y_T = x_j) = b = \frac{\frac{q(e_j)}{\ell(e_j)}}{\sum_{e\in E(x)} \frac{q(e)}{\ell(e)}}.
\end{equation*}
\\
With this in hand, we may also compute from \eqref{eq1} that for $1\leq i\leq k$,

\begin{equation*}
\mathbb{P}^{(i,y)}(Y_T=x_j) = u((i,y)) = \begin{cases} \displaystyle \frac{y}{\ell(e_i)}+\frac{\ell(e_i)-y}{\ell(e_i)}\frac{\frac{q(e_j)}{\ell(e_j)}}{\sum_{e\in E(x)} \frac{q(e)}{\ell(e)}} &\text{if }i=j, \\ \displaystyle \frac{\ell(e_i)-y}{\ell(e_i)}\frac{\frac{q(e_j)}{\ell(e_j)}}{\sum_{e\in E(x)}\frac{q(e)}{\ell(e)}} &\text{if }i\not=j, \\ \displaystyle \frac{\frac{q(e_j)}{\ell(e_j)}}{\sum_{e\in E(x)}\frac{q(e)}{\ell(e)}} &\text{if }i=\text{loop}. \end{cases}
\end{equation*}
\end{proof}

Note that in the above the edge densities $(\omega(e))_{e\in E}$ do not appear.  This is to be expected, given that changing the edge densities only causes a time-change of the Brownian motion $(Y_t)_{t\geq 0}$, which cannot change the probabilities of exiting $B$ through the vertices $x_1,\ldots,x_k$.  Similarly, none of the quantities related to the loop edge $x_{\text{loop}}$ appear, since traversing the loop does not change the probability of exiting $B$ at a specified vertex.

\begin{thm}\label{BMjump1}
The first moment of the hitting time $T$ satisfies

\begin{equation*}
\mathbb{E}^{x}T = \frac{\sum_{e\in E_{\text{loop}}(x)} \omega(e)q(e)\ell(e)}{\sum_{e\in E(x)} \frac{q(e)}{\ell(e)}}.
\end{equation*}
\end{thm}

\begin{proof}
Let $v(y) := \mathbb{E}^yT$.  By general Markov process theory, $v$ satisfies

\begin{equation*}
\begin{cases} \mathcal{L}v = -1&\text{in }B, \\ v = 0 &\text{on }\partial B. \end{cases}
\end{equation*}
\\
Thus, we obtain that $v((i,y)) = -\omega(e_i)y^2+a_iy+b_i$ for some constants $(a_i)_{1\leq i\leq k}$ and $(b_i)_{1\leq i\leq k}$.  Letting $y\downarrow 0$, we get that $b_i = b$ for $1\leq i\leq k$.  Since $v((i,\ell(e_i)) = 0$ for $1\leq i\leq k$, we have that

\begin{equation} \label{eq2}
-\omega(e_i)\ell^2(e_i)+a_i\ell(e_i)+b = 0.
\end{equation}
\\
On the loop edge, since $v((\text{loop},0))=v((\text{loop},\ell(x_{\text{loop}})))=b$, we have that 

\begin{equation*}
v((\text{loop},y))=\omega(x_{\text{loop}})y(\ell(x_{\text{loop}})-y)+b.
\end{equation*}
\\
The compatibility condition \eqref{compcondloop2} gives

\begin{equation*}
\sum^k_{i=1} q(e_i)a_i + q(x_{\text{loop}})\omega(x_{\text{loop}})\ell(x_{\text{loop}}) = 0.
\end{equation*}
\\
As before, we multiply the $i-$th equation in $\eqref{eq2}$ by $\frac{q(e_i)}{\ell(e_i)}$, and sum over $1\leq i\leq k$ to obtain

\begin{equation*}
\mathbb{E}^{x}T = b = \frac{\sum_{e\in E_{\text{loop}}(x)} \omega(e)q(e)\ell(e)}{\sum_{e\in E(x)} \frac{q(e)}{\ell(e)}}.
\end{equation*}
\\
Once we have computed $b$, we may compute from \eqref{eq2} that for $1\leq i\leq k$,

\begin{equation*}
v((i,y)) = -\omega(e_i)y^2+\left(\omega(e_i)\ell(e_i)-\frac{1}{\ell(e_i)}\frac{\sum_{e\in E_{\text{loop}}(x)} \omega(e)q(e)\ell(e)}{\sum_{e\in E(x)} \frac{q(e)}{\ell(e)}}\right)y+\frac{\sum_{e\in E_{\text{loop}}(x)} \omega(e)q(e)\ell(e)}{\sum_{e\in E(x)} \frac{q(e)}{\ell(e)}}.
\end{equation*}
\\
\end{proof}

\begin{thm} \label{BMjump2}
The second moment of the hitting time $T$ and the variance of $T$ are given by

\begin{align*}
\mathbb{E}^{x}T^2  = {}& \frac{1}{3}\frac{\sum_{e\in E_{\text{loop}}(x)} \omega^2(e)q(e)\ell^3(e)}{\sum_{e\in E(x)} \frac{q(e)}{\ell(e)}}+\frac{4}{3}\left(\frac{\sum_{e\in E_{\text{loop}}(x)}  \omega(e)q(e)\ell(e)}{\sum_{e\in E(x)} \frac{q(e)}{\ell(e)}}\right)^2 \\
& +2\omega(x_{\text{loop}})q(x_{\text{loop}})\ell(x_{\text{loop}})\frac{\sum_{e\in E_{\text{loop}}(x)}  \omega(e)q(e)\ell(e)}{\left(\sum_{e\in E(x)} \frac{q(e)}{\ell(e)}\right)^2}, \\
\textup{Var } T = {}& \frac{1}{3}\frac{\sum_{e\in E_{\text{loop}}(x)} \omega^2(e)q(e)\ell^3(e)}{\sum_{e\in E(x)} \frac{q(e)}{\ell(e)}}+\frac{1}{3}\left(\frac{\sum_{e\in E_{\text{loop}}(x)}  \omega(e)q(e)\ell(e)}{\sum_{e\in E(x)} \frac{q(e)}{\ell(e)}}\right)^2 \\
& + 2\omega(x_{\text{loop}})q(x_{\text{loop}})\ell(x_{\text{loop}})\frac{\sum_{e\in E_{\text{loop}}(x)}  \omega(e)q(e)\ell(e)}{\left(\sum_{e\in E(x)} \frac{q(e)}{\ell(e)}\right)^2}.
\end{align*}
\end{thm}
\begin{proof}
Set $w(y) = \mathbb{E}^yT^2$.  Given $\varepsilon\ll 1$ and a point $(i,y)$ which is not one of the vertices $x,x_1,\ldots,x_k$, we use the strong Markov property to obtain

\begin{align*}
w((i,y)) &:= \mathbb{E}^{(i,y)}T^2 = \mathbb{E}^{(i,y)}(\mathbb{E}T^2|\mathcal{F}_\sigma) = \mathbb{E}^{(i,y)}(\mathbb{E}^{Y_\sigma}(T+\sigma)^2) \\
&= \mathbb{E}^{(i,y)}(\mathbb{E}^{Y_\sigma}(T^2+2\sigma T+\sigma^2)) \\
&= \mathbb{E}^{(i,y)}(\mathbb{E}^{Y_\sigma} T^2) + 2\mathbb{E}^{(i,y)}(\sigma \mathbb{E}^{Y_\sigma}T)+\mathbb{E}^{(i,y)}(\mathbb{E}^{Y_\sigma} \sigma^2) \\
&= \frac{1}{2}(w((i,y+\varepsilon))+w((i,y-\varepsilon)))+2\varepsilon^2\omega(e_i)\frac{1}{2}(v((i,y+\varepsilon))+v((i,y-\varepsilon)))+\frac{5}{3}\omega^2(e_i)\varepsilon^4.
\end{align*}

Rearranging as before and letting $\varepsilon\downarrow 0$, we get that $w''((i,y)) = -4\omega(e_i)v((i,y))$, or $\mathcal{L}w = -2v$, so that the function $w$ satisfies

\begin{equation*}
\begin{cases} \mathcal{L}w = -2v&\text{in }B, \\ w = 0 &\text{on }\partial B. \end{cases}
\end{equation*}
\\
The explicit form for $v$ was worked out in the proof of Theorem~\ref{BMjump1}, and hence, using the same techniques as before, one obtains the desired formulae.
\end{proof}

\section{Metrics for the VSRW and for Brownian motion on metric graphs} \label{S3}

In this section, we define and compare several metrics for graphs and metric graphs.  We begin with several conventions.  In this section, we begin with a weighted graph $(\Gamma, (\pi(e))_{e\in E})$ and the associated VSRW $(X_t)_{t\geq 0}$.  We will also analyze the weighted metric graph $\mathcal{X}(\Gamma,\ell,p,\omega)$, which has the associated Brownian motion $(Y_t)_{t\geq 0}$.  The case of weighted metric graphs with loops, which requires virtually no changes from the regular metric graph setting, is discussed at the end of this chapter. \\

Given an edge $e\in E$, we use $\overline{e}$ and $\underline{e}$ to denote the initial and terminal vertices, respectively.  In many situations, the orientation one chooses for edges may not matter (and one may choose an arbitrary orientation); in other situations, such as consecutive edges in a path, there will be a natural orientation. \\

We have already defined geodesic metrics on both $(\Gamma,\pi)$ and $\mathcal{X}(\Gamma,\ell,p,\omega)$.  In our following work, we will need to consider metrics which arise from consideration of the stochastic processes on these spaces.

\subsection{Intrinsic metrics} \label {S3.1}

Both the VSRW on weighted graphs and Brownian motion on metric graphs may be defined through the theory of Dirichlet spaces; such processes have an intrinsic notion of distance associated with them. \\

We work first with the case of Brownian motion on $\mathcal{X}(\Gamma,\ell,p,\omega)$, and view it as the process associated with the strongly local Dirichlet form $(\mathcal{E},\mathcal{D}(\mathcal{E}))$ on $L^2(\mathcal{X}(\Gamma,\ell,p,\omega),\mu)$.  According to Dirichlet form theory (see \cite{St}), there exists a measure $\Gamma(f,g)$ satisfying

\begin{equation*}
\mathcal{E}(f,g) = \int_{\mathcal{X}(\Gamma,\ell)} d\Gamma(f,g). \\
\end{equation*}
\\
Set $L_Y :=\{f\in W^{1,2}_{\text{loc}}(\mathcal{X}(\Gamma,\ell,p,\omega),\mu):d\Gamma(f,f)\leq d\mu\}$.  The intrinsic metric for $Y$ is defined by $\widetilde{d}_I(x,y) := \sup\{f(x)-f(y):f\in L_Y\}$.  For $x\in I(e)$, we have 

\begin{align*}
d\Gamma(f,f)(x) &= (f'(x))^2p(e)dm(x), \\ 
d\mu(x) &= p(e)\omega(e)dm(x). 
\end{align*}
\\
Suppose that $f\in L_Y$.  Consider the edge $e$; if $x\in I(e)$, then $|f'(x)|^2 \leq \omega(e)$, so that $(f(\overline{e})-f(\underline{e}))^2 \leq \omega(e)\ell^2(e)$.  Consequently, if we restrict to points $x,y\in G$, the intrinsic metric satisfies 

\begin{equation*}
\widetilde{d}_I(x,y) := \sup\{|f(x)-f(y)|:f\in\mathcal{E}_Y\},
\end{equation*}
\\
where $\mathcal{E}_Y := \{f\in C(G):\text{for all $e\in E$, }(f(\overline{e})-f(\underline{e}))^2 \leq \omega(e)\ell^2(e)\}$. \\

The situation is slightly more complicated when one considers the VSRW on $(\Gamma,\pi)$.  Here, the associated Dirichlet form is nonlocal.  However, we can make the analogous calculations to those in the local case.  The Dirichlet form is given by

\begin{equation*}
\mathcal{E}(f,g) = \frac{1}{2}\sum_{x,y\in G} \pi_{xy}(f(y)-f(x))^2,
\end{equation*}
\\
and in analogy with the local case we set

\begin{equation*}
\Gamma(f,f)(x) = \frac{1}{2}\sum_{y\sim x} \pi_{xy}(f(y)-f(x))^2,
\end{equation*}
\\
so that

\begin{equation*}
\sum_{x\in G}\Gamma(f,f)(x) = \mathcal{E}(f,f).
\end{equation*}
\\
Set $L_X := \{f\in C(G):\Gamma(f,f)\leq 1\}$ and define the intrinsic metric for $X$ by $d_I(x,y) := \sup\{f(x)-f(y):f\in L_X\}$. Equivalently, we have

\begin{equation*}
d_I(x,y) := \sup\{|f(x)-f(y)|:f\in\mathcal{E}_X\},
\end{equation*}
\\
where $\mathcal{E}_X := \left\{f\in C(G):\text{for all $x\in G$, } \frac{1}{2}\sum_{y\sim x} \pi_{xy}(f(y)-f(x))^2 \leq 1\right\}$.

\subsection{Adapted metrics for the VSRW} \label{S3.2}

We call a metric $\rho$ adapted to the VSRW $X$ on $(\Gamma,\pi)$ if there exists $C_\rho\geq 0$ such that for all $x\in G$,

\begin{equation*}
\sum_{y\sim x} \pi_{xy}\rho^2(x,y) \leq C_\rho.
\end{equation*}
\\
This condition has appeared previously in \cite{F1}, \cite{FLW}, and \cite{GHM}, and it has close apparent connections to the condition in the definition of $\mathcal{E}_X$.  All weighted graphs admit an adapted metric (see the metric $d_V$ in the following section). \\

In a similar vein, suppose that $(\alpha(e))_{e\in E}$ are positive edge weights.  We say that $(\alpha(e))_{e\in E}$ are adapted to the VSRW if there exists $C_\alpha\geq 0$ such that, for each $x\in G$,

\begin{equation*}
\sum_{e\in E(x)} \pi(e)\alpha^2(e) \leq C_\alpha.
\end{equation*}
\\
Any choice of edge weights (even for edge weights which are not adapted) induce metrics in the following ways:

\begin{align*}
d_{1,\alpha}(x,y) &:= \sup\{|f(x)-f(y)|:f\in\mathcal{E}_\alpha\}, \\
d_{2,\alpha}(x,y) &:= \inf\left\{\sum_{e\in\gamma} \alpha(e)^{1/2}:\gamma\text{ is a path joining $x$ and $y$.}\right\},
\end{align*}
\\
where $\mathcal{E}_\alpha := \{f\in C(G):\text{for all $e\in E$, }(f(\overline{e})-f(\underline{e}))^2 \leq \alpha(e)\}$.  Note that for any $x,y\in G$, $d_{1,\alpha}(x,y) = \widetilde{d}_I(x,y)$, where $\widetilde{d}_I$ is the intrinsic metric for Brownian motion $(Y_t)_{t\geq 0}$ on some weighted metric graph $\mathcal{X}(\Gamma,\ell,p,\omega)$ satisfying $\ell(e) = 1$ and $\omega(e)=\alpha(e)$ for each $e\in E$. \\

These metrics coincide:

\begin{lem}
For all $x,y\in G$, $d_{1,\alpha}(x,y)=d_{2,\alpha}(x,y)$.
\end{lem}
\begin{proof}
Pick $x_0\in G$ and set $f_{x_0}(x) := d_{2,\alpha}(x,x_0)$.  Then for $e\in E$,

\begin{equation*}
(f_{x_0}(\overline{e})-f_{x_0}(\underline{e}))^2 \leq d^2_{2,\alpha}(\overline{e},\underline{e}) \leq \alpha(e),
\end{equation*}
\\
so that $f_{x_0}\in \mathcal{E}_\alpha$, and consequently for $x,y\in G$, $d_{1,\alpha}(x,y) \geq |f_x(x)-f_x(y)| =d_{2,\alpha}(x,y)$. \\

For the reverse inequality, given any $\varepsilon>0$ and any $f\in\mathcal{E}_X$, by the definition of $d_{2,\alpha}$, there is a path $\gamma_\varepsilon = (e_0,\cdots,e_n)$ with $\overline{e_0}=x$, $\underline{e_n}=y$ such that

\begin{equation*}
d_{2,\alpha}(x,y)+\varepsilon \geq \sum^{n}_{i=0} \alpha^{1/2}(e_i) \geq \sum^{n}_{i=0} |f(\overline{e_i})-f(\underline{e_i})| \geq |f(x)-f(y)|.
\end{equation*}
\\
Taking the supremum over $f\in\mathcal{E}_X$ and noting that $\varepsilon>0$ was arbitrary, we get that for $x,y\in G$, $d_{2,\alpha}(x,y) \geq d_{1,\alpha}(x,y)$, and consequently $d_{1,a}=d_{2,a}$.
\end{proof}

Consequently, we write $d_\alpha$ to denote the common value of $d_{1,a}$ and $d_{2,a}$, and refer to it as the metric induced by the weights $(\alpha(e))_{e\in E}$.  Since $d_\alpha(\overline{e},\underline{e}) \leq \alpha(e)$, the induced metric associated with a set of adapted edge weights is also adapted. \\

Any metric $\rho$ induces edge weights by setting $\alpha(e) := \rho(e) := \rho(\overline{e},\underline{e})$; if the metric $\rho$ is adapted, then the induced edge weights are also adapted.  In turn, these edge weights induce a metric, which we denote by $d_\rho$.  We refer to $d_\rho$ as the metric induced by the metric $\rho$.  It is not necessarily the case that $\rho$ and $d_\rho$ coincide; for example, if $p>1$ and $d\geq 2$ and $\rho$ is the $p-$norm on the $d-$dimensional lattice $(\mathbb{Z}^d,\mathbb{E}_d)$, then $d_\rho$ is the $1-$norm on $(\mathbb{Z}^d,\mathbb{E}_d)$. \\

\begin{lem} \label{rhodrho}
Let $\rho$ be any metric (not necessarily adapted), and let $d_\rho$ be the induced metric.  Then for all $x,y\in G$, $\rho(x,y) \leq d_\rho(x,y)$.
\end{lem}
\begin{proof}
We have that $d_\rho(x,y) := \sup\{|f(x)-f(y)|:f\in\mathcal{E}_\rho\}$.  Given $x_0\in G$, set $f_{x_0}(x):=\rho(x,x_0)$.  Since

\begin{equation*}
(f_{x_0}(\overline{e})-f_{x_0}(\underline{e}))^2 \leq \rho(\overline{e},\underline{e}) = \rho(e),
\end{equation*}
\\
we conclude that $f_{x_0}\in\mathcal{E}_\rho$, and consequently for $x,y\in G$, $d_{\rho}(x,y) \geq |f_x(x)-f_x(y)| = \rho(x,y)$.
\end{proof}

The intrinsic metric for the VSRW is not, in general, adapted; see Example 2 in Section~\ref{S3.4}. \\

{\bf Remark:} We say that a metric is strongly adapted to the VSRW on $(\Gamma,\pi)$ if there exist positive constants $c_\rho,C_\rho$ such that for each $x\in G$,

\begin{equation*}
c_\rho \leq \sum_{y\sim x} \pi_{xy}\rho^2(x,y) \leq C_\rho.
\end{equation*}
\\
Similarly, we say that the edge weights $(\alpha(e))_{e\in E}$ are strongly adapted to the VSRW on $(\Gamma,\pi)$ if there exist positive constants $c_\alpha,C_\alpha$ such that for each $x\in G$,

\begin{equation*}
c_\alpha \leq \sum_{y\sim x} \pi_{xy}\alpha^2(e) \leq C_\alpha.
\end{equation*}
\\
These conditions will occasionally be useful in our study of stochastic completeness.  In contrast to the condition of adaptedness, strongly adapted metrics (or strongly adapted edge weights) do not always exist for a given weighted graph, as will be shown in a subsequent example.  While strongly adapted metrics induce strongly adapted edge weights, it is not always true that the induced metric for a given set of strongly adapted edge weights is strongly adapted.

\subsection{Other metrics} \label{S3.3}

We define the following additional metrics on $X$:

\begin{align*}
d_E(x,y) &:= \inf\left\{\sum_{e\in\gamma} \pi(e)^{-1/2}:\gamma\text{ is a path joining $x$ and $y$.}\right\}, \\
d_V(x,y) &:= \inf\left\{\sum_{e\in\gamma} \pi_{\overline{e}}^{-1/2}\wedge\pi_{\underline{e}}^{-1/2}:\gamma\text{ is a path joining $x$ and $y$.}\right\}.
\end{align*}
\\
These are the metrics induced by the edge weights $(\pi(e)^{-1/2})_{e\in E}$ and $(\pi_{\overline{e}}^{-1/2}\wedge\pi_{\underline{e}}^{-1/2})_{e\in E}$, respectively.  The letters $E$ and $V$ reflect the fact that the metrics $d_E$ and $d_V$ are constructed using the edge weights $(\pi(e))_{e\in E}$ and the vertex weights $(\pi_x)_{x\in G}$, respectively.  The metric $d_E$ was considered first by Davies in \cite{D}; he used this metric in his study of heat kernel bounds for random walks on graphs .  The metric $d_V$ is a slight modification of a metric introduced independently by the author in \cite{F1}, who used this metric to obtain Gaussian upper bounds for heat kernels of general continuous time simple random walks, and by Grigor'yan, Huang, and Masamune in \cite{GHM}. \\

The following results, which are so simple that we state them without proof, give simple criteria to determine when the metrics $d_E$ and $d_V$ are induced by adapted edge conductances:

\begin{lem} \label{SA1}
Let $(\Gamma,\pi)$ be a weighted graph.  Then the edge weights $(\pi(e)^{-1/2})_{e\in E}$ are strongly adapted (or adapted) to the VSRW if and only if vertex degrees on $\Gamma$ are uniformly bounded.
\end{lem}
\begin{lem} \label{SA2}
Let $(\Gamma,\pi)$ be a weighted graph.  The edge weights $(\pi_{\overline{e}}^{-1/2}\wedge\pi_{\underline{e}}^{-1/2})_{e\in E}$ are adapted to the VSRW.  In particular, the metric $d_V$ is always adapted.  
\end{lem}

\subsection{Comparing metrics}\label{S3.4}

This section clarifies the relationship between the metrics $d_I$, $d_E$, and $d_V$; these estimates are most useful in understanding the intrinsic metric $d_I$, which is difficult to calculate explicitly due to its non-local nature.

\begin{lem}\label{Xcomp1}
For all $x,y\in G$, $d_I(x,y)\leq 2d_E(x,y)$.
\end{lem}
\begin{proof}
It is immediate from the definition of $\mathcal{E}_X$ that whenever $x\sim y$ and $f\in\mathcal{E}_X$, $|f(x)-f(y)|\leq 2\pi_{xy}^{-1/2}$.  For any $\varepsilon>0$ and any $f\in\mathcal{E}_X$, by the definition of $d_E$, there is a path $\gamma_\varepsilon = (x_0,\cdots,x_n)$ such that

\begin{equation*}
d_E(x,y) +\varepsilon \geq \sum^{n-1}_{i=0} \pi^{-1/2}_{x_ix_{i+1}} \geq \frac{1}{2}\sum^{n-1}_{i=0} |f(x_i)-f(x_{i+1})| \geq \frac{1}{2}|f(y)-f(x)|.
\end{equation*}
\\
We conclude that for all $x,y\in G$, $d_I(x,y)\leq 2d_E(x,y)$.
\end{proof}

\begin{lem}\label{Xcomp2}
For all $x,y\in G$, $2^{1/2}d_V(x,y)\leq d_I(x,y)$.
\end{lem}
\begin{proof}
Fix $x_0\in G$.  Note that $f_{x_0}(x) := 2^{1/2}d_V(x,x_0)$ satisfies, for any $x\in G$,

\begin{equation*}
\frac{1}{2}\sum_{y\sim x} \pi_{xy}(f_{x_0}(y)-f_{x_0}(x))^2 \leq \sum_{y\sim x} \pi_{xy}d^2_V(x,y) \leq 1,
\end{equation*}
\\
so that $f_{x_0}\in \mathcal{E}_X$.  Fix $x,y\in G$; we calculate that $d_I(x,y) \geq |f_x(x)-f_x(y)| = 2^{1/2}d_V(x,y)$.
\end{proof}

Now, we impose additional conditions on the structure of $(\Gamma,\pi)$.

\begin{lem}
Suppose that vertex degrees in $(\Gamma,\pi)$ are uniformly bounded by $C_G$.  Then for all $x,y\in G$,

\begin{equation*}
\frac{2}{C_G}d_E(x,y)\leq d_I(x,y) \leq 2d_E(x,y).
\end{equation*}
\end{lem}
\begin{proof}
The rightmost inequality was proven in Lemma~\ref{Xcomp1}.  For the other, we note that for $x_0\in G$, $f_{x_0}(x) := \frac{2}{C_G}d_E(x,x_0)\in\mathcal{E}_X$, which implies that $\frac{2}{C_G}d_E(x,y)\leq d_I(x,y)$.
\end{proof}

\begin{lem}
Suppose that there exists $C>0$ such that for each $e\in E$,

\begin{equation} \label{contweight}
\frac{1}{\pi_e} \leq C\left(\frac{1}{\pi_{\underline{e}}}\vee\frac{1}{\pi_{\overline{e}}}\right).
\end{equation}
\\
Then

\begin{equation*}
2^{1/2} d_V \leq d_I \leq 2d_E \leq 2C^{1/2}d_V.
\end{equation*}
\end{lem}
\begin{proof}
Trivially, \eqref{contweight} implies $d_E\leq C^{1/2}d_V$.  Combining this with Lemma~\ref{Xcomp1} and Lemma~\ref{Xcomp2} gives the desired result. 
\end{proof}

The condition \eqref{contweight} is sometimes referred to as `controlled weights' (see \cite{B}).  It is a stronger condition than uniformly bounded vertex degree. \\

We now give two families of graphs, both of which will be useful in our study of stochastic completeness.  The first family of graphs shows that it is not always possible to find a strongly adapted metric on a graph.  The second family shows that on graphs with unbounded vertex degree, the intrinsic metric $d_I$ may be very different from (for example) the adapted metric $d_V$.  In particular, this example shows that the intrinsic metric may fail to be adapted. \\

{\bf Example 1:} A graph which does not admit a strongly adapted metric. \\

Fix an unbounded function $r:\mathbb{Z}_+\to\mathbb{Z}_+$, and let $(A_n)_{n\in\mathbb{Z}_+}$ be disjoint sets with $|A_n|=r(n)$ and such that $\mathbb{Z}_+ \cap A_n = \varnothing$ for each $n\in\mathbb{Z}_+$.  Let $\Gamma_r := (G_r, E_r)$ be as follows:

\begin{align*}
G_r &:= \mathbb{Z}_+ \cup \bigcup_{n\in\mathbb{Z}_+} A_n, \\
E_r &:= \bigcup_{n\in\mathbb{Z}_+}\bigcup_{x\in A_n} \{n,x\}\cup \bigcup_{n\in\mathbb{Z}_+}\bigcup_{x\in A_n} \{n+1,x\}.
\end{align*}
\\
Equip this graph with the standard weights.  Suppose that $\Gamma_r$ admits a strongly adapted metric $\rho$.  Then there exist positive constants $c_\rho,C_\rho$ such that for each $x\in G_r$,

\begin{equation*}
c_\rho \leq \sum_{e\in E(x)} \pi(e)\rho^2(\overline{e},\underline{e}) \leq C_\rho.
\end{equation*}
\\
On the other hand, for each $n\in\mathbb{Z}_+$,

\begin{align*}
2C_r &\geq \sum_{e\in E(n)} \pi(e)\rho^2(\overline{e},\underline{e}) + \sum_{e\in E(n+1)} \pi(e)\rho^2(\overline{e},\underline{e}) \\
&\geq \sum_{x\in A_n}\sum_{e\in E(x)} \pi(e)\rho^2(\overline{e},\underline{e}) \\
&\geq r(n)c_r,
\end{align*}
\\
which is impossible since $r$ is unbounded. \\

Note also that this graph is always stochastically complete, regardless of the choice of $r$.   Every other jump is to a vertex in some $A_n$, and at such vertices, since there are only two neighbors, the mean jump time of the VSRW is $1/2$.  Consequently, the lifetime of the VSRW on this graph is infinite.  This example shows that there exist stochastically complete graphs with arbitrarily large volume growth (this has been observed earlier using a different construction in \cite{Wo}).   In particular, it is not possible to obtain results for general graphs showing that sufficiently large volume growth implies stochastic incompleteness. \\

{\bf Example 2:} Spherically symmetric trees with increasing rate of branching. \\

Let $\Gamma_\alpha$ $(0<\alpha<2)$ be a tree rooted at $x_0$, with all vertices at a graph distance of $r$ from $x_0$ having $k(r):= \lfloor r^\alpha \rfloor$ neighbors at a graph distance of $r+1$ from $x_0$; we equip these graphs with the standard weights. \\

We compute that if $x_R\in\Gamma_\alpha$ satisfies $d(x_0,x_R)=R$, then

\begin{align*}
d_V(x_0,x_R) \asymp \sum^{R}_{j=1} \frac{1}{j^{\alpha/2}} \asymp R^{1-\alpha/2},
\end{align*}
\\
Let $\gamma := (x_0,\ldots,x_R)$ be the geodesic (minimal length) path joining $x_0$ and $x_R$; we use $V(\gamma)$ to denote the set of vertices in $\gamma$.  For $0\leq j\leq R$, we set $f(x_j)=j/2$.  For $x\in\Gamma_\alpha\setminus V(\gamma)$, let $n(x)$ denote the unique $y\in V(\gamma)$ satisfying $d(x,V(\gamma))=d(x,y)$, and set $f(x) := f(n(x))$.  It is clear that $f\in\mathcal{E}_X$, and hence that $d_I(x_0,x_R)\geq R/2$. \\

In particular, this example shows that in general, one cannot find a constant $C>0$ such that $d_V(x,y) \geq Cd_I(x,y)$, even if we restrict to $x,y\in G$ such that $d(x,y) \geq R$ for some constant $R>0$.  We will revisit this setting in Example 2 of Section~\ref{S5.1}.

\subsection{Measures}\label{S3.5}

Consider the general continuous time simple random walk on $(\Gamma,(\pi(e))_{e\in E},(\theta_x)_{x\in G})$ given by

\begin{equation*}
(\mathcal{L}_\theta f)(x) := \frac{1}{\theta_x}\sum_{y\sim x}\pi_{xy}(f(y)-f(x)).
\end{equation*}
\\
This process is a time-change of the VSRW, depending on the choice of the vertex measure $(\theta_x)_{x\in G}$.  The invariant measure associated with this process is simply $(\theta_x)_{x\in G}$; hence for the VSRW we have that the measure of $U\subset G$ is

\begin{equation*}
m_G(U) := |U|.
\end{equation*}
\\
We have already defined the measure associated with Brownian motion on $\mathcal{X}(\Gamma,\ell,p,\omega)$; given $V\subset\mathcal{X}(\Gamma,\ell)$, we have that

\begin{equation*}
m_{MG}(V):=\int_V \mu(dx).
\end{equation*}
\\
Notably, for an edge $e$, we have

\begin{equation*}
m_{MG}(e)= \int_{I(e)} \sum_{f\in E} {\bf 1}_{I(f)} p(f)\omega(f)m(dx) = \omega(e)p(e)\ell(e).
\end{equation*}

\subsection{Loops}\label{S3.6}

Here we work with Brownian motion $(Y_t)_{t\geq 0}$ on $\mathcal{X}(\Gamma,\ell,p,\omega)$ and Brownian motion $(Z_t)_{t\geq 0}$ on the augmented graph $\mathcal{X}_{\text{loop}}(\Gamma,\ell,p,\omega)$.  We have already noted that the Dirichlet space formulation of Brownian motion goes through unchanged when loops are added; in particular, if $m_{MG_{\text{loop}}}$ denotes the measure on $\mathcal{X}_{\text{loop}}(\Gamma,\ell,p,\omega)$, then the restriction of $m_{MG_{\text{loop}}}$ to subsets of $\mathcal{X}(\Gamma,\ell)$ is equal to $m_{MG}$; as such, there is no need to distingush between $m_{MG}$ and $m_{MG_{\text{loop}}}$.  As before, for any $e\in E_{\text{loop}}$, we have $m_{MG_{\text{loop}}}(e) = \omega(e)p(e)\ell(e)$. \\

Loops also have very little impact on the intrinsic metrics.  Let $\widetilde{d}_{I,Y}$ denote the intrinsic metric on $\mathcal{X}(\Gamma,\ell,p,\omega)$ and $\widetilde{d}_{I,Z}$ denote the intrinsic metric on $\mathcal{X}_{\text{loop}}(\Gamma,\ell,p,\omega)$.  As before, we have that for $x,y\in \mathcal{X}(\Gamma,\ell)$ and $u,v\in \mathcal{X}_{\text{loop}}(\Gamma,\ell)$,

\begin{align*}
\widetilde{d}_{I,Y}(x,y) &:= \sup\{|f(x)-f(y)|:f\in\mathcal{E}_Y\}, \\
\widetilde{d}_{I,Z}(u,v) &:= \sup\{|g(u)-g(v)|:g\in\mathcal{E}_Z\},
\end{align*}
\\
where

\begin{align*}
\mathcal{E}_Y &:= \{f\in W^{1,2}_{\text{loc}}(\mathcal{X}(\Gamma,\ell,p,\omega),\mu):\text{for all $e\in E$, if $x\in I(e)$, }|f'(x)|^2\leq \omega(e)\}, \\
\mathcal{E}_Z &:= \{f\in W^{1,2}_{\text{loc}}(\mathcal{X}_{\text{loop}}(\Gamma,\ell,p,\omega),\mu):\text{for all $e\in E_\text{loop}$, if $x\in I(e)$, }|f'(x)|^2\leq \omega(e)\}.
\end{align*}

\begin{lem}\label{metricloop}
If $x,y\in\mathcal{X}(\Gamma,\ell)$, then $\widetilde{d}_{I,Y}(x,y) = \widetilde{d}_{I,Z}(x,y)$.
\end{lem}
\begin{proof}
First, the restriction of any $g\in \mathcal{E}_Z$ to $\mathcal{X}(\Gamma,\ell)$ is in $\mathcal{E}_Y$; as an immediate consequence $\widetilde{d}_{I,Z}(x,y) \leq \widetilde{d}_{I,Y}(x,y)$.  On the other hand, given $f\in\mathcal{E}_Y$, one may extend $f$ to $\mathcal{X}_{\text{loop}}(\Gamma,\ell)$ by setting, at each $x\in G$, $f(y)=f(x)$ for each $y\in I(x_{\text{loop}})$.  This extension is in $\mathcal{E}_Z$; consequently $\widetilde{d}_{I,Y}(x,y) \leq \widetilde{d}_{I,Z}(x,y)$, and hence $\widetilde{d}_{I,Y}(x,y) = \widetilde{d}_{I,Z}(x,y)$ for all $x,y\in\mathcal{X}(\Gamma,\ell)$.  
\end{proof}

In particular, $\widetilde{d}_{I,Y}$ and $\widetilde{d}_{I,Z}$ are equal when evaluated between points of $G$.

\section{Synchronizing Brownian motion on metric graphs with the VSRW} \label{S4}

In this section, we begin with a weighted graph $(\Gamma, (\pi(e))_{e\in E})$ and the associated VSRW $(X_t)_{t\geq 0}$.  We investigate the question of when it is possible to find a weighted metric graph $\mathcal{X}(\Gamma,\ell,p,\omega)$ or $\mathcal{X}_{\text{loop}}(\Gamma,\ell,p,\omega)$ such that Brownian motion $(Y_t)_{t\geq 0}$ on this space behaves similarly to the VSRW $(X_t)_{t\geq 0}$ with respect to certain properties. \\

Let $T$ be the hitting time for $Y$ described in Section~\ref{S2}; it is the time it takes $Y$ to hit a vertex different from the last visited one.  Let $\widetilde{T}$ be the jump time of the VSRW.  We write $\mathbb{P}^x$ for the law of $(Y_t)_{t\geq 0}$ started at $x\in G$, and $\widetilde{\mathbb{P}}^x$ for the law of $(X_t)_{t\geq 0}$ started at $x\in G$. \\

Our first result concerns jump probabilities only:

\begin{thm}\label{syncjump}
$X$ and $Y$ have the same jump probabilities if and only if $(\ell(e))_{e\in E}$ and $(p(e))_{e\in E}$ are chosen such that there exists $\lambda>0$ such that for each $e\in E$, $\lambda \pi(e) = p(e)/\ell(e)$. \\
\end{thm}
\begin{proof}
Fix $x\in G$.  Setting $\mathbb{P}^{x}(Y_T= x_j) = \widetilde{\mathbb{P}}^{x}(X_{\widetilde{T}} = x_j)$, we obtain

\begin{equation*}
\frac{\frac{p(e_j)}{\ell(e_j)}}{\sum_{e\in E(x)} \frac{p(e)}{\ell(e)}}=\frac{\pi(e_j)}{\sum_{e\in E(x)} \pi(e)},
\end{equation*}
\\
This implies that there exists $\lambda=\lambda(x)$ such that $\lambda \pi(e) = p(e)/\ell(e)$ for $e\in E(x)$.  Given an edge $e:=\{u,v\}$, by comparing the equalities $\mathbb{P}^{u}(Y_T = x_j) = \widetilde{\mathbb{P}}^{u}(X_{\widetilde{T}} = x_j)$ and $\mathbb{P}^{v}(Y_T = x_j) = \widetilde{\mathbb{P}}^{v}(X_{\widetilde{T}} = x_j)$, we obtain that $\lambda(u)=\lambda(v)$.  Since $G$ is connected, we conclude that $\lambda$ is constant.   
\end{proof}

The situation for synchronizing expected jump times is more complex.  We have the following necessary and sufficient criteria which shows when it is possible to synchronize the jump probabilities and expected jump times using a loopless metric graph: \\

We begin with a definition from graph theory: \\

A {\it disjoint cycle cover} of $\Gamma = (G,E)$ is a collection of vertex-disjoint cycles in $G$ such that every vertex in $G$ is incident to some edge in one of the cycles.  A single edge is considered a cycle.

\begin{thm}\label{sync1}
It is possible to choose $(\ell(e))_{e\in E}$, $(p(e))_{e\in E}$, and $(\omega(e))_{e\in E}$ so that $X$ and $Y$ satisfy, for each $x\in G$,

\begin{align*}
\mathbb{P}^{x}(Y_T = x_j) &= \widetilde{\mathbb{P}}^{x}(X_{\widetilde{T}} = x_j), \\
\mathbb{E}^xT &= \widetilde{\mathbb{E}}^{x}\widetilde{T}.
\end{align*}
\\
if and only, if for every $e\in E$, there exists a disjoint cycle cover of $(G,E)$ which uses the edge $e$. \\
\end{thm}
\begin{proof}
First, by the previous result, we know that there exists $\lambda>0$ such that $(p(e))_{e\in E}$ and $(\ell(e))_{e\in E}$ satisfy $p(e)/\ell(e) = \lambda \pi(e)$ for each $e\in E$. \\

Fix $x\in G$.  Setting $\mathbb{E}^xT = \widetilde{\mathbb{E}}^{x}\widetilde{T}$, we get that

\begin{equation*}
\frac{\sum_{e\in E(x)} \omega(e)p(e)\ell(e)}{\sum_{e\in E(x)} \frac{p(e)}{\ell(e)}} = \frac{1}{\sum_{e\in E(x)} \pi(e)}.
\end{equation*}
\\
Given that $(p(e))_{e\in E}$ and $(\ell(e))_{e\in E}$ have been chosen, we set $\omega(e) := c(e)(p(e)\ell(e))^{-1}$; we then get that

\begin{equation*}
\frac{\sum_{e\in E(x)} c(e)}{\sum_{e\in E(x)} \pi(e)} = \frac{1}{\sum_{e\in E(x)} \pi(e)}.
\end{equation*}
\\
In other words, the problem is to determine when it is possible to assign edge weights $(c(e))_{e\in E}$ to each edge of the graph so that the edge weights incident to each vertex sum to $1$.  By the following lemma, we see that this happens if and only if, for each $e\in E$, there is a disjoint cycle cover containing $e$ in one of its cycles.
\end{proof}

\begin{lem}
Suppose that $\Gamma=(G,E)$ is a locally finite graph.  It is possible to assign edge weights to each edge of the graph $(G,E)$ so that the edge weights incident to each vertex sum to $1$ if and only if, for each $e\in E$, there is a disjoint cycle cover containing $e$ in one of its cycles. \\
\end{lem}
\begin{proof}
We reproduce the proof at \cite{MO}.  Suppose that for each $e\in E$, there is a disjoint cycle cover containing $e$ in one of its cycles.  Fix an edge $f$, and pick a disjoint cycle cover containing that edge.  For each $e\in E$, we define $c_f(e) := 1$ if the edge $f$ appears as an isolated edge in the disjoint cycle cover, $c_f(e) := 1/2$ if the edge $e$ is part of a proper cycle in the disjoint cycle cover, and $c_f(e) = 0$ otherwise.  For any $x\in G$, we have $\sum_{e\in E(x)} c_f(e) = 1$, and $c_f(f)>0$, but not necessarily that $c_f(e)>0$ for all edges $e$. \\

Now, let $(\alpha(e))_{e\in E}$ be any collection of positive numbers satisfying $\sum_{e\in E} \alpha(e) = 1$, and set, for each $e\in E$, $c(e) = \sum_{f\in E} \alpha(f)c_f(e)$. \\

Note that since $\alpha(f)c_f(e)>0$, $c(e)>0$, and for any vertex $x\in G$ with neighbors $e_1,\ldots,e_k$, we have that

\begin{align*}
\sum_{e\in E(x)} c(e) &= \sum_{e\in E(x)}\sum_{f\in E}\alpha(f)c_f(e) \\
&= \sum_{f\in E}\sum_{e\in E(x)} \alpha(f)c_{f}(e) \\
&= \sum_{f\in E}\alpha(f) \\
&= 1
\end{align*}
\\
as desired. \\

On the other hand, suppose that for each $x\in G$, the sum of the edge weights incident to $x$ is $1$.  It is clear that the (possibly infinite) weighted adjacency matrix $A$ is doubly stochastic, and by the Birkhoff-von Neumann theorem is a (possibly infinite) convex combination of permutation matrices, which we denote by $(P_n)_{n\in\mathcal{I}}$.  Write

\begin{equation*}
A = \sum_{n\in\mathcal{I}} a_nP_n.
\end{equation*}
\\
with $a_n>0$, $\sum_{n\in\mathcal{I}} a_n = 1$.  Fix $e\in E$.  Then the corresponding entry of $A$ is nonzero, so some $P_j$ must have a nonzero entry in the same place.  $A$ has zero entries on its diagonal (as $(G,E)$ has no loops), so each $P_j$ must also.  By considering the cycle decomposition of the permutation corresponding to the matrix $P_j$, we see that $P_j$ naturally corresponds to a disjoint cycle cover containing the edge $e$.
\end{proof}

{\bf Remarks:}  1. The condition arising in Theorem~\ref{sync1} is very unstable, as it can be destroyed by perturbing a graph very slightly; modifying any graph so that it has at least one vertex of degree $1$ will make it impossible to synchronize the VSRW with a Brownian motion as in Theorem~\ref{sync1}. \\

2. It is not difficult to see that it is always possible to synchronize the jump probabilities and jump times on $\mathcal{X}_{\text{loop}}(\Gamma,\ell,p,\omega)$; one proceeds as above and sets $c(e) := \frac{1}{2}(\pi_{\overline{e}}^{-1/2}\wedge \pi_{\underline{e}}^{-1/2})$ for all non loop edges, so that $\sum_{e\in E(x)} c(e) \leq \frac{1}{2}$ for all $x\in G$.  One then chooses the edge weights on loops so that $\omega(x_{\text{loop}})p(x_{\text{loop}})\ell(x_{\text{loop}}) = 1-\sum_{e\in E(x)} c(e) \leq \frac{1}{2}$ for all $x\in G$.  However, the intrinsic metric for the Brownian motion that one obtains from this process may have undesirable properties. \\

In applications, we may be studying the VSRW on $(\Gamma,\pi)$ with a particular adapted metric in mind.  The following result allows us to construct Brownian motion on a weighted metric graph $\mathcal{X}_{\text{loop}}(\Gamma,\ell,p,\omega)$ such that the VSRW and the Brownian motion have the same jump probabilities and approximately the same expected jump times, and with the additional property that the intrinsic metric for the Brownian motion is closely related to the adapted metric for the VSRW.

\begin{thm} \label{sync2}
Let $(\Gamma,\pi)$ be a weighted graph with adapted edge weights $(\alpha(e))_{e\in E}$.  Let $\mathcal{X}(\Gamma,\ell,p,\omega)$ be the metric graph such that, for $e\in E$,

\begin{align*}
\ell(e) &:= 1, \\
p(e) &:= \pi(e), \\
\omega(e) &:= \alpha^2(e),
\end{align*}
\\
and for each $x\in G$,

\begin{equation*}
\ell(x_{\text{loop}}) := 1, \ p(x_{\text{loop}}) := \frac{1}{2}, \ \omega(x_{\text{loop}}):=1.
\end{equation*}
\\
There exists $C_\alpha>0$ such that for any $x\in G$,

\begin{align}
\mathbb{P}^{x}(Y_T = x_j) &= \widetilde{\mathbb{P}}^{x}(X_{\widetilde{T}} = x_j), \label{jump}\\
\widetilde{\mathbb{E}}^{x}\widetilde{T} &\leq \mathbb{E}^xT \leq  (C_\alpha+1)\widetilde{\mathbb{E}}^{x}\widetilde{T}. \label{moment}
\end{align}
\end{thm}
\begin{proof}
The relation \eqref{jump} follows immediately from Theorem~\ref{syncjump}.  By adaptedness of $(\alpha(e))_{e\in E}$, there exists a positive constant $C_\alpha$ such that for any $x\in G$,

\begin{equation*}
0 \leq \sum_{e\in E(x)} \pi_{e}\alpha^2(e) \leq C_\alpha,
\end{equation*}
\\
from which it follows immediately that

\begin{equation} \label{SA}
1 \leq \sum_{e\in E_{\text{loop}}(x)} \omega(e)p(e)\ell(e) \leq C_\alpha+1,
\end{equation}
\\
so that \eqref{moment} follows from \eqref{SA} and \eqref{BMjump1}.
\end{proof}
\begin{cor} \label{sync3}
Let $(\Gamma,\pi)$ be a weighted graph, and let $\rho$ be a metric adapted to the VSRW.  Let $\mathcal{X}(\Gamma,\ell,p,\omega)$ be the metric graph such that, for $e\in E$,

\begin{align*}
\ell(e) &:= 1, \\
p(e) &:= \pi(e), \\
\omega(e) &:= \rho^2(\overline{e},\underline{e}).
\end{align*}
\\
and for each $x\in G$,

\begin{equation*}
\ell(x_{\text{loop}}) := 1, \ p(x_{\text{loop}}) := \frac{1}{2}, \ \omega(x_{\text{loop}}):=1.
\end{equation*}
\\
 There exist positive constants $c_\rho, C_\rho$ such that for any $x\in G$,

\begin{align}
\mathbb{P}^{x}(Y_T = x_j) &= \widetilde{\mathbb{P}}^{x}(X_{\widetilde{T}} = x_j),\\
\widetilde{\mathbb{E}}^{x}\widetilde{T} &\leq \mathbb{E}^xT \leq (C_\alpha+1)\widetilde{\mathbb{E}}^{x}\widetilde{T}.
\end{align}
\\
Additionally, if $d_\rho$ is the metric induced by the edge weights $\rho(e) := \rho(\overline{e},\underline{e})$, then for all $x,y\in G$, $\rho(x,y) \leq d_\rho(x,y)$.
\end{cor}
\begin{proof}
This follows immediately from Lemma~\ref{rhodrho}, Lemma~\ref{metricloop}, and Theorem~\ref{sync2}.
\end{proof}

{\bf Remarks:} 1. Theorem~\ref{sync2} is our first result where use of the augmented graph $\mathcal{X}_{\text{loop}}(\Gamma,\ell,p,\omega)$ is essential.  If one wishes to prove an analogue of this result on the loopless metric graph $\mathcal{X}(\Gamma,\ell,p,\omega)$ while still having control over the intrinsic metric for the associated Brownian motion, then the necessary condition on the edge weights $(\alpha(e))_{e\in E}$ is not adaptedness but rather strong adaptedness.  This is undesirable since strongly adapted edge weights do not exist for every weighted graph (c.f. Example 1 of Section~\ref{S3.4}). \\

2. Note that the choice that $\ell(e) = 1$ for all $e\in E$ is somewhat arbitrary.  Take any $\varphi:E\to (0,+\infty)$.  If one replaces $(\ell(e))_{e\in E}$, $(p(e))_{e\in E}$, $(\omega(e))_{e\in E}$ with $(\ell_\varphi(e))_{e\in E}$, $(p_\varphi(e))_{e\in E}$, $(\omega_\varphi(e))_{e\in E}$ satisfying $\ell_\varphi(e) := \ell(e)\varphi(e)$, $p_\varphi(e) :=\ell(e)(\varphi(e))^{-1}$, $\omega_\varphi(e) := \omega(e)(\varphi(e))^2$, then the Brownian motion on the new weighted graph behaves identically to the Brownian motion on the original weighted graph with respect to hitting probabilities and moments of hitting times. \\

\section{Stochastic completeness of graphs} \label{S5}

In this section we present general results relating volume growth in adapted metrics to stochastic completeness of $(\Gamma,\pi)$, and corollaries specializing to the intrinsic metric $d_I$ and the adapted metric $d_V$.  These results are analogous to Grigor'yan's result on stochastic completeness on manifolds, and will be seen to produce sharp results when applied to specific graphs for which exact criteria for stochastic completeness are known.  Additionally, we present an example showing that in contrast to the setting of local Dirichlet spaces, the intrinsic metric $d_I$ may be a poor metric for analyzing stochastic completeness of graphs.\\

We begin by compiling two results on the convergence of random series which will be used in the proof of our main results.

\subsection{Convergence of random series}

In this section, we let $(\Omega,\mathcal{F},\mathbb{P})$ be a probability space, on which we have the sequence of random variables $(X_n)_{n\in\mathbb{Z}_+}$, and the associated filtration $\mathcal{F}_n:=\sigma(X_0,\ldots,X_n)$.  We define the following events:

\begin{align*}
\mathcal{A} &:= \left\{\sum^\infty_{n=0} X_n < +\infty\right\}, \\
\mathcal{B} &:= \left\{\sum^\infty_{n=0} \mathbb{P}(|X_n|>1|\mathcal{F}_{n-1}) <+\infty\right\}, \\
\mathcal{C} &:= \left\{\sum^\infty_{n=0} \mathbb{E}(X_n{\bf 1}_{\{X_n\leq 1\}}|\mathcal{F}_{n-1}) <+\infty\right\}, \\
\mathcal{D} &:= \left\{\sum^\infty_{n=0} \textup{Var}(X_n{\bf 1}_{\{X_n\leq 1\}}|\mathcal{F}_{n-1}) <+\infty\right\}.
\end{align*}

Two events are said to be equivalent if their symmetric difference has probability $0$.  We have the following two results:

\begin{thm} \label{Doob}
(Doob, \cite{Do})  Suppose that the sequence of random variables $(X_n)_{n\in\mathbb{Z}_+}$ is nonnegative and that there exists $C>0$ such that $\mathbb{P}(0\leq X_n\leq C) = 1$ for all $n\in\mathbb{Z}_+$.  Then the events $\mathcal{A}$ and $\mathcal{C}$ are equivalent.
\end{thm}

\begin{thm} \label{Brown}
(Brown, \cite{Br})  The event $\mathcal{B}\cap\mathcal{C}\cap\mathcal{D}$ is almost surely a subset of $\mathcal{A}$ (i.e., $\mathbb{P}((\mathcal{B}\cap\mathcal{C}\cap\mathcal{D})\setminus\mathcal{A})=0$).  That is, if each of the series associated with $\mathcal{B}$, $\mathcal{C}$, and $\mathcal{D}$ converge, then the series associated with $\mathcal{A}$ converges also.
\end{thm}

{\bf Remark:} The paper \cite{Br} claims to prove a generalization of the Kolmogorov three-series theorem for dependent random variables.  While \cite{Br} does correctly prove that simultaneous convergence of the series associated with the events $\mathcal{B}$, $\mathcal{C}$, and $\mathcal{D}$ is sufficient for the convergence of the original series (which yields Theorem~\ref{Brown}), the proof of necessity is incorrect, and in \cite{Gi} it is shown that there can be no general three-series theorem of this type. 

\subsection{Proof of the main result}

Our proof uses the following criterion of Sturm which relates stochastic completeness to volume growth on local Dirichlet spaces:

\begin{thm} \label{Sturm}
(Sturm, \cite{St})  Let $(Y_t)_{t\geq 0}$ be the stochastic process associated with the strongly local Dirichlet form $(\mathcal{E},\mathcal{D}(\mathcal{E}))$ on $L^2(X,m)$, and denote the intrinsic metric on this space by $\rho$.  Suppose that there exists $x_0\in X$ and $r_0>0$ such that

\begin{equation*}
\int^\infty_{r_0} \frac{r}{\log m(B_\rho(x_0,r))}dr = + \infty.
\end{equation*}
\\
Then $(Y_t)_{t\geq 0}$ is non-explosive.
\end{thm}

In this section, $T$ and $\widetilde{T}$ will denote jump times for $Y$ and $X$ as in previous sections. 

\begin{thm} \label{SC1}
Let $(\Gamma,\pi)$ be a weighted graph, and let  $(\alpha(e))_{e\in E}$ be edge weights adapted to the VSRW on $(\Gamma,\pi)$ such that there exists $D_\alpha>0$ satisfying $\alpha(e)\leq D_\alpha$ for all $e\in E$.  Let $\rho$ be any metric satisfying $\rho(x,y) \leq d_\alpha(x,y)$.  If there exists $x_0\in G$ and $r_0>0$ such that

\begin{equation*}
\int^\infty_{r_0} \frac{r}{\log m_G(B_\rho(x_0,r))}dr = + \infty,
\end{equation*}
\\
then $(\Gamma,\pi)$ is stochastically complete.  
\end{thm}
\begin{proof}
Let $(X_t)_{t\geq 0}$ denote the VSRW on $(\Gamma,\pi)$, which has the natural filtration $(\mathcal{G}^X_t)_{t\geq 0}$, let $(\sigma_n)_{n\in\mathbb{Z}_+}$ denote the jump times for $X$, and set $H_X(n) := \sum^n_{j=0} \sigma_j$ and $\mathcal{F}^X_n := \mathcal{G}^X_{H_X(n)}$. \\

We begin with the following lemma: \\

\begin{lem}\label{SClem1}
Under the hypotheses of Theorem~\ref{SC1}, $\mathbb{P}-$a.s.,

\begin{equation*}
\sum^\infty_{n=0} \mathbb{E}(\sigma_n|\mathcal{F}^X_{n-1}) = +\infty.
\end{equation*}
\end{lem}
\begin{proof}
By adaptedness, there exists $C_\alpha>0$ such that for all $x\in G$,

\begin{equation*}
\sum_{e\in E(x)} \pi(e)\alpha^2(e) \leq C_\alpha.
\end{equation*}
\\
We will work with the metric graph $\mathcal{X}_{\text{loop}}(\Gamma,\ell,p,\omega)$, where for $e\in E$,

\begin{align*}
\ell(e) &:= 1, \\
p(e) &:= \pi(e), \\
\omega(e) &:= \alpha^2(e),
\end{align*}
\\
and for each $x\in G$,

\begin{equation*}
\ell(x_{\text{loop}}) := 1, p(x_{\text{loop}}) := \frac{1}{2}, \omega(x_{\text{loop}}) := 1.
\end{equation*}
\\
Let $X$ denote the VSRW on $(\Gamma,\pi)$ and let $Y$ denote Brownian motion on $\mathcal{X}_{\text{loop}}(\Gamma,\ell,p,\omega)$.  By Theorem~\ref{sync2}, $X$ and $Y$ have the same jump probabilities.  We begin with the process $(Y_t)_{t\geq 0}$.  By sampling $Y$ each time it hits a vertex different from the last visited vertex, we obtain a discrete time simple random walk on $(\Gamma,\pi)$, which we denote by $(Z_n)_{n\in\mathbb{Z}_+}$, and we use this simple random walk to construct the VSRW $(X_t)_{t\geq 0}$; this is a coupling of $X$ and $Y$; in particular $X$ and $Y$ exist on the same probability space, so there is no need to distinguish between $\mathbb{P}$ and $\widetilde{\mathbb{P}}$ (or $\mathbb{E}$ and $\widetilde{\mathbb{E}}$).  Let $(\sigma_n)_{n\in\mathbb{Z}_+}$ and $(\tau_n)_{n\in\mathbb{Z}_+}$ be the jump times for $X$ and $Y$, respectively; setting $H_X(n) := \inf\{t\geq H_X(n-1):X_t\in G\setminus\{Z_{n-1}\}\}$ and $H_Y(n) := \inf\{t\geq H_Y(n-1):Y_t\in G\setminus\{Z_{n-1}\}\}$, we have that

\begin{align*}
\sigma_n &:= H_X(n)-H_X(n-1), \\
\tau_n &:= H_Y(n)-H_Y(n-1).
\end{align*}
\\
Let $(\mathcal{F}^X_n)_{n\in\mathbb{Z}_+}$ be the filtration such that for each $n\in\mathbb{Z}_+$, $\mathcal{F}^X_n:=\mathcal{G}^X_{H_X(n)}$, where $(\mathcal{G}^X_t)_{t\geq 0}$ is the natural filtration for $X$.  Similarly, let $(\mathcal{F}^Y_n)_{n\in\mathbb{Z}_+}$ be the filtration such that for each $n\in\mathbb{Z}_+$, $\mathcal{F}^Y_n:=\mathcal{G}^Y_{H_Y(n)}$, where $(\mathcal{G}^Y_t)_{t\geq 0}$ is the natural filtration for $Y$. \\

From Theorem~\ref{sync2}, for $n\in\mathbb{Z}_+$, 

\begin{equation} \label{tausigma}
\mathbb{E}(\sigma_n|\mathcal{F}^X_{n-1}) \leq \mathbb{E}(\tau_n|\mathcal{F}^Y_{n-1}) \leq (C_\alpha+1)\mathbb{E}(\sigma_n|\mathcal{F}^X_{n-1}), 
\end{equation}
\\
and for $x,y\in G$, $\rho(x,y) \leq d_\alpha(x,y)$.  By Lemma~\ref{metricloop}, $d_\alpha$ agrees with $\widetilde{d}_I$ when evaluated between points of $G$.  Thus, for all $r\geq 0$, $B_\rho(x_0,r)\supseteq G\cap B_{\widetilde{d}_I}(x_0,r)$.  We have that

\begin{align*}
m_G(B_\rho(x_0,r)) &:= |B_\rho(x_0,r)|, \\
m_{MG}(B_{\widetilde{d}_I}(x_0,r)) &\leq \sum_{e\in E_{\text{loop}}\text{ such that }e\cap B_{\widetilde{d}_I}(x_0,r)\not=\varnothing} m_{MG}(e) \\
&\leq \sum_{x\in G\cap B_{\widetilde{d}_I}(x_0,r)} \sum_{e\in E_{\text{loop}}(x)} m_{MG}(e) \\
&= \sum_{x\in G\cap B_{\widetilde{d}_I}(x_0,r)} \sum_{e\in E_{\text{loop}}(x)} \omega(e)p(e)\ell(e) \\
&\leq \sum_{x\in G\cap B_{\widetilde{d}_I}(x_0,r)} (C_\alpha+1) \\
&\leq (C_\alpha+1)|B_\rho(x_0,r)|.
\end{align*}
\\
In particular, since $(C_\alpha+1)m_G(B_\rho(x_0,r))\geq m_{MG}(B_{\widetilde{d}_I}(x_0,r))$, the hypothesis

\begin{equation*}
\int^\infty_{r_0} \frac{r}{\log m_G(B_\rho(x_0,r))}dr = + \infty
\end{equation*}
\\
implies that

\begin{equation*}
\int^\infty_{r_0} \frac{r}{\log m_{MG}(B_{\widetilde{d}_I}(x_0,r))}dr = + \infty.
\end{equation*}
\\
By Theorem~\ref{Sturm}, this implies non-explosiveness of $Y$, and hence, $\mathbb{P}-$a.s.,

\begin{equation*}
\sum^\infty_{n=0} \tau_n = +\infty.
\end{equation*}
\\
By Theorem~\ref{Brown}, for any $K>0$, $\mathbb{P}-$a.s., at least one of the following equalities holds:

\begin{align}
\sum^\infty_{n=0} \mathbb{P}(\tau_n\geq K|\mathcal{F}^Y_{n-1}) &= +\infty, \label{TS1} \\
\sum^\infty_{n=0} \mathbb{E}(\tau_n{\bf 1}_{\{\tau_n\leq K\}}|\mathcal{F}^Y_{n-1}) &= +\infty, \label{TS2} \\
\sum^\infty_{n=0} \text{Var}(\tau_n{\bf 1}_{\{\tau_n\leq K\}}|\mathcal{F}^Y_{n-1}) &= +\infty. \label{TS3}
\end{align}
\\
Fix $K>0$.  We will show that any of \eqref{TS1}, \eqref{TS2}, \eqref{TS3} implies

\begin{equation}
\sum^\infty_{n=0} \mathbb{E}(\tau_n|\mathcal{F}^Y_{n-1}) = +\infty. \label{Esigma}
\end{equation}
\\
If \eqref{TS1} holds, by Markov's inequality,
\begin{align*}
\sum^\infty_{n=0} \mathbb{P}(\tau_n\geq K|\mathcal{F}^Y_{n-1}) &\leq \frac{1}{K}\sum^\infty_{n=0} \mathbb{E}(\tau_n{\bf 1}_{\{\tau_n \geq K\}}|\mathcal{F}^Y_{n-1}) \\
&\leq \frac{1}{K}\sum^\infty_{n=0} \mathbb{E}(\tau_n|\mathcal{F}^Y_{n-1}).
\end{align*}
\\
If \eqref{TS2} holds, we use the trivial estimate

\begin{equation*}
\sum^\infty_{n=0} \mathbb{E}(\tau_n{\bf 1}_{\{\tau_n\leq K\}}|\mathcal{F}^Y_{n-1}) \leq \sum^\infty_{n=0} \mathbb{E}(\tau_n|\mathcal{F}^Y_{n-1}).
\end{equation*}
\\
If \eqref{TS3} holds, we begin by noting that

\begin{align*}
\text{Var}(\tau_n|\mathcal{F}^Y_{n-1})-\text{Var}(\tau_n{\bf 1}_{\{\tau_n\leq K\}}|\mathcal{F}^Y_{n-1}) = {} &\mathbb{E}(\tau^2_n|\mathcal{F}^Y_{n-1})-\mathbb{E}(\tau^2_n{\bf 1}_{\{\tau_n\leq K\}}|\mathcal{F}^Y_{n-1}) \\
& -(\mathbb{E}(\tau_n|\mathcal{F}^Y_{n-1}))^2+(\mathbb{E}(\tau_n{\bf 1}_{\{\tau_n\leq K\}}|\mathcal{F}^Y_{n-1}))^2 \\
= {} &\mathbb{E}(\tau_n^2{\bf 1}_{\{\tau_n>K\}}|\mathcal{F}^Y_{n-1})+(\mathbb{E}(\tau_n{\bf 1}_{\{\tau_n\leq K\}}|\mathcal{F}^Y_{n-1}))^2 \\
& -(\mathbb{E}(\tau_n|\mathcal{F}^Y_{n-1}))^2 \\
\geq {} &(\mathbb{E}(\tau_n{\bf 1}_{\{\tau_n>K\}}|\mathcal{F}^Y_{n-1}))^2+(\mathbb{E}(\tau_n{\bf 1}_{\{\tau_n\leq K\}}|\mathcal{F}^Y_{n-1}))^2 \\
& -(\mathbb{E}(\tau_n|\mathcal{F}^Y_{n-1}))^2 \\
\geq {} &-\frac{1}{2}(\mathbb{E}(\tau_n|\mathcal{F}^Y_{n-1}))^2.
\end{align*}
\\
Hence

\begin{equation*}
\text{Var}(\tau_n|\mathcal{F}^Y_{n-1})+\frac{1}{2}(\mathbb{E}(\tau_n|\mathcal{F}^Y_{n-1}))^2 \geq \text{Var}(\tau_n{\bf 1}_{\{\tau_n\leq K\}}|\mathcal{F}^Y_{n-1}).
\end{equation*}
\\
We conclude that \eqref{TS3} implies that one of the following equalities holds:

\begin{align}
\sum^\infty_{n=0} \text{Var}(\tau_n|\mathcal{F}^Y_{n-1}) &= +\infty, \label{Varsigma} \\
\sum^\infty_{n=0} (\mathbb{E}(\tau_n|\mathcal{F}^Y_{n-1}))^2 &= +\infty. \label{E2sigma}
\end{align}
\\
Suppose that \eqref{Varsigma} holds.  Note that if $\tau_n$ is the jump time for $(Y_t)_{t\geq 0}$ started at the vertex $x$, then by Theorem~\ref{BMjump2},

\begin{align*}
\text{Var}(\tau_n|Y_{H_Y(n-1)}=x) = {}& \frac{1}{3}\frac{\sum_{e\in E_{\text{loop}}(x)} \omega^2(e)q(e)\ell^3(e)}{\sum_{e\in E(x)} \frac{q(e)}{\ell(e)}}+\frac{1}{3}\left(\frac{\sum_{e\in E_{\text{loop}}(x)}  \omega(e)q(e)\ell(e)}{\sum_{e\in E(x)} \frac{q(e)}{\ell(e)}}\right)^2 \\
& + 2\omega(x_{\text{loop}})q(x_{\text{loop}})\ell(x_{\text{loop}})\frac{\sum_{e\in E_{\text{loop}}(x)}  \omega(e)q(e)\ell(e)}{\left(\sum_{e\in E(x)} \frac{q(e)}{\ell(e)}\right)^2} \\
\leq {}& \frac{1}{3}\frac{\sum_{e\in E_{\text{loop}}(x)} \omega^2(e)q(e)\ell^3(e)}{\sum_{e\in E(x)} \frac{q(e)}{\ell(e)}}+\frac{7}{3}(\mathbb{E}(\tau_n|Y_{H_Y(n-1)}=x))^2 \\
\leq {}& \frac{1}{3\pi_x}\left(D^2_\alpha\sum_{e\in E(x)} \omega(e)p(e)\ell(e)+1\right)+\frac{7}{3}(\mathbb{E}(\tau_n|Y_{H_Y(n-1)}=x))^2 \\
\leq {}&\frac{1}{3}(D^2_\alpha\vee 1)\mathbb{E}(\tau_n|Y_{H_Y(n-1)}=x)+\frac{7}{3}(\mathbb{E}(\tau_n|Y_{H_Y(n-1)}=x))^2,
\end{align*}
\\
and hence

\begin{equation*}
\text{Var}(\tau_n|\mathcal{F}^Y_{n-1}) \leq \frac{1}{3}(D^2_\alpha\vee 1)\mathbb{E}(\tau_n|\mathcal{F}^Y_{n-1})+\frac{7}{3}(\mathbb{E}(\tau_n|\mathcal{F}^Y_{n-1}))^2.
\end{equation*}
\\
From this estimate, it is clear that if \eqref{Varsigma} holds, then either \eqref{Esigma} holds, in which case we are done, or \eqref{E2sigma} holds.  In the latter case, either the positive sequence $(\mathbb{E}(\tau_n|\mathcal{F}^Y_{n-1}))_{n\in\mathbb{Z}_+}$ converges to $0$, in which case eventually $\mathbb{E}(\tau_n|\mathcal{F}^Y_{n-1})\geq (\mathbb{E}(\tau_n|\mathcal{F}^Y_{n-1}))^2$ (implying that that \eqref{Esigma} holds), or it does not converge to $0$, in which case it is clear that \eqref{Esigma} holds also.  Thus, we conclude that \eqref{TS3} implies \eqref{Esigma}. \\

We conclude that $\sum^\infty_{n=0} \tau_n = +\infty$ $\mathbb{P}-$a.s. implies $\sum^\infty_{n=0} \mathbb{E}(\tau_n|\mathcal{F}^Y_{n-1}) = +\infty$ $\mathbb{P}-$a.s.  By \eqref{tausigma}, this implies $\sum^\infty_{n=0} \mathbb{E}(\sigma_n|\mathcal{F}^X_{n-1}) = +\infty$ $\mathbb{P}-$a.s. also.
\end{proof}

At this point, we assume that the vertex weights $(\pi_x)_{x\in G}$ are bounded below by $C_\pi>0$.  We will subsequently discharge this assumption by a probabilistic argument. \\

\begin{lem}\label{SClem2}
Under the hypotheses of Theorem~\ref{SC1} and the additional hypothesis that the vertex weights $(\pi_x)_{x\in G}$ are bounded below by $C_\pi>0$, $(\Gamma,\pi)$ is stochastically complete.
\end{lem}
\begin{proof}
By Lemma~\ref{SClem1}, we have that $\mathbb{P}$-a.s.,

\begin{equation*}
\sum^\infty_{n=0} \mathbb{E}(\sigma_n|\mathcal{F}^X_{n-1}) = +\infty.
\end{equation*}
\\
Note that $\mathbb{E}(\sigma_n|X_{H_X(n-1)}=x)=\frac{1}{\pi_x}$ and that the jump time from the vertex $x$ is exponential with parameter $\pi_x$.  Clearly, $\pi_x \geq C_\pi>0$.  Then we have that

\begin{align*}
\mathbb{E}(\sigma_n{\bf 1}_{\{\sigma_n\leq 1\}}|Y_{H_Y(n-1)}=x) &= \int^1_0 \pi_xue^{-\pi_{x}u}du \\
&= \frac{1}{\pi_{x}}(1-(\pi_{x}+1)e^{-\pi_{x}}) \\
&\geq \frac{1}{\pi_{x}}(1-(C_\pi+1)e^{-C_\pi}) \\
&\geq \frac{C'_\pi}{\pi_{x}} \\
&= C'_\pi \mathbb{E}(\sigma_n|X_{H_X(n-1)}=x).
\end{align*}
\\
Note that $C'_\pi := 1-(C_\pi+1)e^{-C_\pi} = \int^{C_\pi}_0 ve^{-v}dv > 0$.  It follows that for $n\in\mathbb{Z}_+$,

\begin{equation*}
\mathbb{E}(\sigma_n{\bf 1}_{\{\sigma_n\leq 1\}}|\mathcal{F}^X_{n-1}) \geq C'_\pi \mathbb{E}(\sigma_n|\mathcal{F}^X_{n-1}),
\end{equation*}
\\
and hence $\mathbb{P}-$a.s., $\sum^\infty_{n=0} \mathbb{E}(\sigma_n{\bf 1}_{\{\sigma_n\leq 1\}}|\mathcal{F}^X_{n-1}) = +\infty$.  Since $(\sigma_n{\bf 1}_{\{\sigma_n\leq 1\}})_{n\in\mathbb{Z}_+}$ is a uniformly bounded sequence of nonnegative random variables, Theorem~\ref{Doob} implies that $\mathbb{P}-$a.s.,

\begin{equation*}
\sum^\infty_{n=0} \sigma_n{\bf 1}_{\{\sigma_n\leq 1\}} = +\infty.
\end{equation*}

Since $0\leq \sigma_n{\bf 1}_{\{\sigma_n\leq 1\}} \leq \sigma_n$ pointwise, we conclude that $\mathbb{P}-$a.s.,

\begin{equation*}
\sum^\infty_{n=0} \sigma_n = +\infty,
\end{equation*}
\\
and hence $(\Gamma,\pi)$ is stochastically complete.
\end{proof}
Finally, we remove the hypothesis that the vertex weights are bounded below.  Suppose that $(\Gamma,\pi)$ is a stochastically incomplete graph which satisfies the hypotheses of Theorem~\ref{SC1}, but has vertex weights which are not bounded below. \\

Let $H := \{x\in G:\pi_x\leq 1\}$.  Consider the augmented graph $(\widetilde{\Gamma},\widetilde{\pi})$, where $\widetilde{\Gamma} = (\widetilde{G},\widetilde{E})$ is obtained from $(G,E)$ by adding, for each $x\in H$, a vertex $\widetilde{x}$ connected to the rest of the graph only by the edge $\{x,\widetilde{x}\}$; denote these additional vertices by $\widetilde{H}$.  The edges of the form $\{x,\widetilde{x}\}$ are given weight $1$. \\

Let $(\widetilde{X}_t)_{t\geq 0}$ denote the VSRW on $(\widetilde{\Gamma},\widetilde{\pi})$.  We use this to construct the VSRW on $(\Gamma,\pi)$ via a coupling, as follows:  At each vertex in $G\setminus H$, $\widetilde{X}$ and $X$ move identically.  At a vertex $x\in H$, $\widetilde{X}$ eventually jumps to some vertex in $\widetilde{G}\setminus\{x,\widetilde{x}\}$.  $X$ jumps to the same vertex after waiting an exponential time with parameter $\pi_x$. \\

We define $\widetilde{\rho}$ on $(\widetilde{\Gamma},\widetilde{\pi})$ as follows.  Given $x,y\in G$ with $x\not=y$, we set

\begin{align*}
\widetilde{\rho}(x,y) &= \rho(x,y), \\
\widetilde{\rho}(x,\widetilde{y}) &= \rho(x,y)+1, \\
\widetilde{\rho}(\widetilde{x},\widetilde{y}) &= \rho(x,y)+2.
\end{align*}
\\
Since $\rho$ was adapted and $\pi_{x\widetilde{x}}\widetilde{\rho}^2(x,\widetilde{x}) = 1$ for all $x\in H$, $\widetilde{\rho}$ is an adapted metric on $(\widetilde{\Gamma},\widetilde{\pi})$ with constant at most $C_\rho+1$.  \\

For any $x_0\in G$ and $r\geq 0$, we have that 

\begin{equation*}|B_\rho(x_0,r)| \leq |B_{\widetilde{\rho}}(x_0,r)| \leq |B_\rho(x_0,r)|+|\{\widetilde{x}\in \widetilde{H}:d_\rho(x_0,x)\leq r\}| \leq 2|B_\rho(x_0,r)|.
\end{equation*}
\\
Consequently, $(\widetilde{\Gamma},\widetilde{\pi})$ satisfies the hypotheses of Lemma~\ref{SClem2}, so $(\widetilde{\Gamma},\widetilde{\pi})$ is stochastically complete and $\widetilde{X}$ is non-explosive. \\

By hypothesis, $X$ explodes with positive probability, and we denote the event on which this occurs by $SI_X$.  Given $U\subset G$ and $V\subset \widetilde{G}$, $T^X_U$ is the total amount of time $X$ spends at vertices of $U$, and similarly $T^{\widetilde{X}}_V$ is the total amount of time $\widetilde{X}$ spends at vertices of $V$. \\

On $SI_X$, we have that

\begin{equation*}
T^X_{G\setminus H} + T^X_H = T^X_G < \infty.
\end{equation*}
\\
Clearly $T^X_{G\setminus H} < \infty$ on $SI_X$.  Since $X$ and $\widetilde{X}$ move identically on $G\setminus H$, this implies that $T^{\widetilde{X}}_{G\setminus H} < \infty$ on $SI_X$. \\

Next, we use the fact that $T^X_H < \infty$ on $SI_X$.  Since the times between jumps for the VSRW are independent of the jump directions, the times between jumps from vertices in $H$ are independent exponentially distributed random variables with mean at least $1$.  By the second Borel-Cantelli Lemma, an infinite number of visits to $H$ would cause $X$ to spend an infinite amount of time at vertices of $H$, and have infinite lifetime, which would be impossible.  So on $SI_X$, $X$ visits $H$ only finitely many times. \\

Using the coupling, each visit to a vertex $x\in H$ by $X$ is associated with an initial visit to $x$ by $\widetilde{X}$, as well as a geometrically distributed (finite) number of additional visits to $x$ and $\widetilde{x}$ by $\widetilde{X}$ (before $\widetilde{X}$ visits $\widetilde{G}\setminus\{x,\widetilde{x}\}$ again).  The times between all of these jumps for $\widetilde{X}$ are independent exponentially distributed random variables with mean at most $1$; consequently, their sum is finite.  Since $\widetilde{X}$ only visits $H$ and $\widetilde{H}$ finitely many times on $SI_X$, we conclude that on $SI_X$, $T^{\widetilde{X}}_H<\infty$ and $T^{\widetilde{X}}_{\widetilde{H}}<\infty$. \\

Thus, on $SI_X$,

\begin{equation*}
T^{\widetilde{X}}_{\widetilde{G}} = T^{\widetilde{X}}_{G\setminus H}+T^{\widetilde{X}}_{H}+T^{\widetilde{X}}_{\widetilde{H}} < \infty.
\end{equation*}
\\
Consequently, $\widetilde{X}$ explodes with positive probability, contradicting the stochastic completeness of $(\widetilde{\Gamma},\widetilde{\pi})$. \\

Combining this argument with Lemma~\ref{SClem2}, we conclude that under the hypotheses of Theorem~\ref{SC1}, $(\Gamma,\pi)$ is stochastically complete.
\end{proof}
\begin{cor}\label{cor1}
Let $(\Gamma,\pi)$ be a weighted graph, and let $\rho$ be a metric adapted to the VSRW on $(\Gamma,\pi)$ such that there exists $D_\rho>0$ satisfying $\rho(x,y)\leq D_\rho$ whenever $x\sim y$.  If there exists $x_0\in G$ and $r_0>0$ such that

\begin{equation*}
\int^\infty_{r_0} \frac{r}{\log|B_\rho(x_0,r)|}dr = + \infty,
\end{equation*}
\\
then $(\Gamma,\pi)$ is stochastically complete.
\end{cor}
\begin{proof}
Upon noting that the edge conductances $\rho(e) := \rho(\overline{e},\underline{e})$ are adapted and uniformly bounded above, this follows from Corollary~\ref{sync3} and Corollary~\ref{SC1}.
\end{proof}
\begin{cor}\label{cor2}
Let $(\Gamma,\pi)$ be a weighted graph with uniformly bounded vertex degrees and edge weights uniformly bounded below.  If there exists $x_0\in G$ and $r_0>0$ such that

\begin{equation*}
\int^\infty_{r_0} \frac{r}{\log |B_{d_I}(x_0,r)|}dr = + \infty,
\end{equation*}
\\
then $(\Gamma,\pi)$ is stochastically complete. \\
\end{cor}
\begin{proof}
Set $\alpha(e) := \pi(e)^{-1/2}$ for $e\in E$.  By Lemma~\ref{SA1} and the hypothesis of edge weights uniformly bounded below, the hypotheses of Theorem~\ref{SC1} are satisfied, and by Theorem~\ref{Xcomp1}, for $x,y\in G$, $d_I(x,y) \leq d_E(x,y) = d_\alpha(x,y)$, so the result follows from Theorem~\ref{SC1}.
\end{proof}
\begin{cor} \label{cor3}
Let $(\Gamma,\pi)$ be a weighted graph such that the vertex weights $(\pi_x)_{x\in G}$ are bounded below.  If there exists $x_0\in G$ and $r_0>0$ such that 

\begin{equation*}
\int^\infty_{r_0} \frac{r}{\log|B_{d_V}(x_0,r)|}dr=+\infty,
\end{equation*}
\\
then $(\Gamma,\pi)$ is stochastically complete. 
\end{cor}
\begin{proof}
This is an immediate consequence of Lemma~\ref{SA2} and Corollary~\ref{cor1}.
\end{proof}

Given a weighted graph $(\Gamma,\pi)$, we define $d_W$ to be the metric induced by the edge weights $(1\wedge \pi_{\overline{e}}^{-1/2}\wedge\pi_{\underline{e}}^{-1/2})_{e\in E}$.  It is clear that $d_W \leq d_V$, so $d_W$ is always adapted, and also that $d_W(x,y) \leq 1$ whenever $x\sim y$.  

\begin{cor} \label{cor4}
Let $(\Gamma,\pi)$ be a weighted graph.  If there exists $x_0\in G$ and $r_0>0$ such that 

\begin{equation*}
\int^\infty_{r_0} \frac{r}{\log|B_{d_W}(x_0,r)|}dr=+\infty,
\end{equation*}
\\
then $(\Gamma,\pi)$ is stochastically complete. 
\end{cor}
\begin{proof}
This is an immediate consequence of Corollary~\ref{cor1}.
\end{proof}

{\bf Remarks:}  1.  The analogue of Sturm's result relating stochastic completeness with volume growth in the intrinsic metric is false; see Example 2 in Section~\ref{S5.1} for an example where the VSRW is stochastically incomplete but where $\log |B_{d_I}(x,r)|\asymp r\log r$.  On graphs with unbounded vertex degree and sufficiently poor connectivity properties (in the sense that there are very few geodesic paths joining vertices), the intrinsic metric is not adapted to the VSRW.  Indeed, in Section~\ref{S3.4} we gave an example in which the VSRW moves very quickly but the intrinsic metric is approximately equal to the graph metric. Consequently, it appears that the intrinsic metric is not the right metric to use for studying stochastic completeness of graphs, and that adaptedness or strong adaptedness is the relevant condition, particularly in light of Corollary~\ref{cor1}.  It would be interesting to determine conditions under which the intrinsic metric is adapted to the VSRW for graphs with unbounded vertex degree. \\

2.  None of the Theorems and Corollaries in this section use the notion of strong adaptedness.  However, the lower bound appearing in the definition of strong adaptedness ensures that the metric cannot be unnecessarily small, which would in turn cause the volume growth to be unnecessarily large.  If the metric $\rho$ fails to be strongly adapted, the resulting criteria one obtains for stochastic completeness may be far from optimal.  See the remark following Example 2 in Section~\ref{S5.1}. \\

3.  These techniques and results are also applicable to the general continuous time simple walk with generator $\mathcal{L}_\theta$ defined in Section~\ref{S3.5}; call this process $(Z_t)_{t\geq 0}$.  Here one works with edge weights $(\beta(e))_{e\in E}$ which are bounded above and adapted to this random walk, i.e., weights satisfying, for all $x\in G$,

\begin{equation*}
\frac{1}{\theta_x}\sum_{e\in E(x)}\pi(e)\beta^2(e) \leq C_\beta.
\end{equation*}
\\
One constructs the associated Brownian motion as in Theorem~\ref{sync2}, and sets $\ell(x_{\text{loop}})=1$, $p(x_{\text{loop}})=\frac{1}{2}\theta_x$, $\omega(x_{\text{loop}})=1$ for $x\in G$.  The techniques of this section allow one to prove the analogue of Theorem~\ref{SC1} for $(Z_t)_{t\geq 0}$, using the invariant measure described in Section~\ref{S3.5} and considering volume growth in any metric $\rho$ such that $\rho(x,y) \leq d_\beta(x,y)$ for all $x,y\in G$.  However, this criteria may not be very useful.  For example, if $\theta_x = \pi_x$, the associated random walk (which is referred to as the constant-speed continuous time simple random walk, or CSRW, and which jumps at exponentially distributed times with mean $1$) is always stochastically complete, regardless of volume growth.

\subsection{Examples} \label{S5.1}

In this section, we present several examples which demonstrate how one can use Theorem~\ref{SC1} and its various corollaries to quickly determine sharp stochastic completeness criteria for various graphs. \\

{\bf Example 1:} Stochastic completeness of the birth-death chain.  We have $\Gamma := (\mathbb{Z}_+,E_{nn})$, where $E_{nn}:=\{\{n,n+1\}:n\in\mathbb{Z}_+\}$, and we assign weights $(\pi(e))_{e\in E}$ satisfying $\pi_{n,n+1} = (n+1)^2\log_+^\beta(n+1)$ for $0\leq \beta<2$, where $\log_+(x) := \log(x)\vee 1$.  We use the metric $d_V$ for ease of computation.  As $R\to\infty$,

\begin{equation*}
d_V(0,R) \sim \sum^{R}_{j=1} \frac{1}{2j\log_+^{\beta/2}(j)} \sim \frac{1}{2}\log^{1-\beta/2}(R),
\end{equation*}
\\
and hence $|B_{d_V}(0,r)| \asymp e^{(2r)^{2/(2-\beta)}}$.  By Corollary~\ref{cor2} or Corollary~\ref{cor3}, we conclude that if $0\leq \beta\leq 1$, then $(\Gamma,\pi)$ is stochastically complete.  This result may be seen to be sharp; by Example 4.12 of \cite{Wo}, $(\Gamma,\pi)$ is stochastically incomplete if and only if

\begin{equation} \label{SCBD}
\sum^\infty_{r=0} \frac{r}{\pi_{r,r+1}} < +\infty.
\end{equation}
\\
Plugging in $\pi_{n,n+1} = (n+1)^2\log_+^\beta(n+1)$ with $0\leq \beta<2$, we see that \eqref{SCBD} occurs precisely when $1<\beta<2$. \\

{\bf Example 2:}  Stochastic completeness of spherically symmetric trees. \\

The construction of these objects was done in Example 2 in Section~\ref{S3.4}.  Let $\Gamma_\alpha$ be a tree rooted at $x_0$, with all vertices at a graph distance of $r$ from $x_0$ having $k(r):= \lfloor r^\alpha \rfloor$ neighbors at a graph distance of $r+1$ from $x_0$ for $0<\alpha<2$.  We equip these graphs with the standard weights.  Clearly vertex degrees are unbounded in this setting. \\

As before, we use the adapted metric $d_V$ for ease of computation.  We previously computed that if $x_R\in\Gamma_\alpha$ satisfies $d(x_0,x_R)=R$, then $d_V(x_0,x_R) \asymp R^{1-\alpha/2}$.  Hence if $d_V(x_0,y)\asymp r$, then $d(x_0,y)\asymp r^{2/(2-\alpha)}$.  As well, $|B_d(x_0,r)|=\sum^r_{j=0}\prod^{j}_{i=0}k(i)$, so that

\begin{align*}
\log |B_{d_V}(x_0,r)| \asymp \sum^{r^{2/(2-\alpha)}}_{j=1} \log j^\alpha \asymp r^{2/(2-\alpha)}\log r.
\end{align*}
\\ 
From Corollary~\ref{cor2}, we conclude that $\Gamma_\alpha$ is stochastically complete if $\alpha\leq 1$.  By Remark 4.3 of \cite{Wo}, the exponent $1$ is sharp; if $\alpha>1$ then $\Gamma_\alpha$ is stochastically incomplete. \\

Note that in this setting the metric $d_V$ is strongly adapted.  Now, suppose that we work on $\Gamma_1$ with a different choice of metric.  Given $1<\beta<2$, we consider the metric $d_\beta$ induced by the edge weights $c_\beta(e)=r^{-\beta/2}$ if $d(x_0,\overline{e})\wedge d(x_0,\underline{e})=r$; these weights are adapted but not strongly adapted.  Proceeding as above, we obtain that $\log |B_{d_\beta}(x_0,r)| \asymp r^{2/(2-\beta)}\log r$; this gives

\begin{equation*}
 \int^\infty_{r_0} \frac{r}{\log |B_{d_\beta}(x_0,r)|}dr < +\infty,
\end{equation*}
\\
even though $\Gamma_1$ is stochastically complete.  Consequently, this adapted metric does not give sharp volume growth criteria. \\

This family of graphs also yields an interesting counterexample.  We previously computed that $d_I(x_0,x_R) \asymp R$, and it follows that for any $\alpha>1$

\begin{equation*}
\log |B_{d_I}(x_0,r)| \asymp \sum^{r}_{j=1} \log j^\alpha \asymp r\log r,
\end{equation*}
\\ 
from which it follows that

\begin{equation*}
\int^\infty_{r_0} \frac{r}{\log|B_{d_I}(x_0,r)|}dr=+\infty,
\end{equation*}
\\
even though $\Gamma_\alpha$ is stochastically incomplete.  Comparing this result with Theorem~\ref{Sturm} shows that the relationship between volume growth in the intrinsic metric and stochastic completeness on graphs is different than the corresponding result for local Dirichlet spaces. \\

{\bf Acknowledgements: }The author thanks the following individuals for their detailed reading of the manuscript and helpful comments.  Martin Barlow suggested this problem, answered many questions relating to the theory of Dirichlet spaces, and provided detailed feedback on many revisions of this paper.  Ed Perkins gave several useful suggestions.  Pat Fitzsimmons informed the author of the error in the main result of \cite{Br}, which rendered the original proof of Lemma~\ref{SClem2} invalid.

\end{document}